\documentclass[10pt, a4paper]{article}

\usepackage{geometry}
\usepackage{amsmath}
\usepackage{amsthm}
\usepackage{amssymb}
\usepackage{bbm} 
\usepackage{color}
\usepackage{xcolor}
\usepackage{graphicx}
\usepackage[toc,page]{appendix}
\usepackage{authblk}
\usepackage[unicode,colorlinks=true,linkcolor=black,citecolor=black,urlcolor=black,pagebackref=false]{hyperref} 

\newcommand{\R}{\mathbb{R}}
\newcommand{\1}{\mathbbm{1}}
\renewcommand{\P}{\mathbb{P}}
\newcommand{\F}{\mathcal{F}}
\newcommand{\E}{\mathbb{E}}

\newtheorem{theorem}{Theorem}
\newtheorem{proposition}[theorem]{Proposition}
\newtheorem{lemma}[theorem]{Lemma}
\newtheorem{corollary}[theorem]{Corollary}

\newtheorem{assumption}[theorem]{Assumption}
\theoremstyle{definition}
\newtheorem{example}[theorem]{Example}

\title{Stochastic Volterra equations: failure of the time-homogeneous Markov property}

\author[1]{Martin Friesen \thanks{Email: martin.friesen@dcu.ie}}
\author[2]{Stefan Gerhold \thanks{Email: sgerhold@fam.tuwien.ac.at}}
\author[2]{Kristof Wiedermann \thanks{Email: kristof.wiedermann@fam.tuwien.ac.at}}
\affil[1]{\small School of Mathematical Sciences, Dublin City University}
\affil[2]{\small Institute of Statistics and Mathematical Methods in Economics, TU Wien}

\date{\today}
\numberwithin{equation}{section}
\numberwithin{theorem}{section}

\begin{document}

\maketitle

\begin{abstract}
\noindent Path-dependence is a defining feature of many real-world systems, with applications ranging from population dynamics to rough volatility models and electricity spot prices. In stochastic Volterra equations (SVEs), such dependence is encoded in the Volterra kernel, which dictates how past trajectories influence present dynamics on infinitesimal time scales. This structure suggests a breakdown of the Markov property. In this article, we develop computational techniques and methods based on small-time asymptotics for SVEs with Hölder coefficients to rigorously establish that they cannot possess the time-homogeneous Markov property. In particular, for affine drifts, we characterise the time-homogeneous Markov property and show that under natural non-degeneracy conditions, it only holds for exponential Volterra kernels. 
\end{abstract}
\vspace{0.2cm}
{\small \textbf{Keywords:} stochastic Volterra equation; path-dependence; Markov property.\vspace{0.2cm}\newline
}

\section{Introduction}

In recent years, stochastic Volterra equations (SVEs) have become a highly active area of research, with a major driving force coming from the rapidly developing field of rough volatility models (see \cite{rvol, pricingroughvol, volisrough}). These models offer remarkable flexibility as they can reproduce the empirically observed roughness of realized volatilities and simultaneously provide a tractable framework for modeling the small-time at-the-money volatility skew. However, this flexibility comes at the cost of dealing with more sophisticated, path-dependent dynamics, which, in the case of singular kernels, even fall outside the classical setting of semimartingales.

A natural consequence of this path-dependence is the breakdown of the Markov property. Intuitively, this can be traced back to the Volterra structure itself, where convolutions with the Volterra kernel encode a non-trivial dependence on the past trajectory. This phenomenon is not merely of theoretical interest; it also plays a central role in, e.g., numerical schemes for simulating SVEs, where Euler-type discretization schemes (see~\cite{alfonsi2022}), or splitting methods (see~\cite{alfonsi2023}) depend on the whole trajectory of the process. Other schemes that incorporate the memory and provide Markovian approximations of SVEs have been studied in \cite{multifactorapproxroughvol} and \cite{MR4521278}. As for the scarce literature on \emph{non}-Markovianity of stochastic processes, let us mention the recent work~\cite{ChZh25}, where the failure of the Markov property is studied for a process involving standard Brownian motion and its concave majorant. Despite the widespread belief that stochastic Volterra processes are not Markov and the formal recognition of this fact in the literature, e.g.\ at many places of \cite{rvol}, a rigorous proof has so far remained absent. 

As an illustration, consider the rough Cox-Ingersoll-Ross process, the cornerstone of affine Volterra processes in rough volatility modeling. It describes the stochastic instantaneous variance in terms of the unique nonnegative continuous weak solution, see \cite[Theorem~6.1]{AbiJaLaPu19}, of the stochastic Volterra equation
\begin{equation}\label{eq:roughHeston}
    X_t = x_0 + \int_0^t \frac{(t-s)^{H - 1/2}}{\Gamma(H + 1/2)}\kappa( \theta - X_s)\, \mathrm{d}s 
    + \sigma \int_0^t \frac{(t-s)^{H - 1/2}}{\Gamma(H + 1/2)} \sqrt{X_s}\, \mathrm{d}B_s,\quad t\in\R_+,
\end{equation}
where $x_0 \geq 0$ denotes the initial instantaneous variance, $\kappa \geq 0$ the mean-reversion speed, $\theta \geq 0$ the long-term mean of the process, $\sigma > 0$ the vol-of-vol, $H \in (0,1/2)$ the \textit{Hurst parameter}, and $B$ a standard Brownian motion, see \cite{roughhestonhedging} and \cite{roughhestcharfct}. Here, it is natural to expect that the failure of the Markov property stems from the path-dependence introduced via the convolution with the fractional Riemann-Liouville kernel $t^{H-1/2}/\Gamma(H + 1/2)$, $t\in\R_+^*$. Generalised Langevin processes provide another interesting class of Volterra processes that appear in applications in mathematical physics (see \cite{MR3704863}).

In this paper, we study a general class of stochastic Volterra equations (SVEs) on a closed state-space $D \subseteq \R^d$ of the form
\begin{equation}\label{eq:generalSVIE}
  X_t = g(t;x) + \int_{0}^{t}K_b(t-s) \hspace{0.03cm}b(X_{s})\, \mathrm{d}s + \int_0^t K_{\sigma}(t-s)\hspace{0.03cm}\sigma(X_s)\, \mathrm{d}B_s, \quad t\in\R_+,
\end{equation} 
where $b: D \longrightarrow \R^d$ and $\sigma: D \longrightarrow \R^{d\times m}$ correspond to the drift and diffusion coefficient, $B$ is an $m$-dimensional standard Brownian motion, $K_b \in L_{\mathrm{loc}}^1(\R_+; \R^{d \times d})$, $K_{\sigma} \in L_{\mathrm{loc}}^2(\R_+; \R^{d \times d})$ denote the Volterra kernels, and $g(\cdot\hspace{0.05cm};x): \R_+ \longrightarrow \R^d$ is the initial curve of the dynamics with $g(0;x)=x\in D$. Clearly, without any further assumption, the existence and uniqueness of solutions are not guaranteed, but have been thoroughly studied in e.g.\ \cite{weak_solution,10.1214/24-EJP1234,AbiJaLaPu19,CF24,proemel_scheffels_weak}. For continuous coefficients $b$, $\sigma$ satisfying linear growth conditions, \cite[Theorem 3.4]{AbiJaLaPu19} provides weak existence for~\eqref{eq:generalSVIE} when $g(\cdot\hspace{0.05cm};x)\equiv x$, $D=\R^d$ and $K_b = K_{\sigma}$ admit a resolvent of the first kind and satisfy mild power bounds in $L^2_{\mathrm{loc}}$. This covers, e.g., completely monotone kernels. When $g(\cdot\hspace{0.05cm};x)\in C(\R_+; \R^{d})$ is sufficiently regular, weak existence to \eqref{eq:generalSVIE} can also be established via \cite[Theorem 1.2 and Theorem 6.1]{weak_solution}, \cite[Theorem 2.7]{CF24} and \cite[Theorem~3.3]{proemel_scheffels_weak} going beyond the completely monotone case. Finally, weak existence for the Jacobi Volterra process on $D=[\alpha_1,\alpha_2]$ with $\alpha_1\le \alpha_2$ can be established via \cite[Corollary~2.8]{10.1214/24-EJP1234}. 

Suppose that for each $x \in D$, \eqref{eq:generalSVIE} admits a continuous weak solution. Let $(\P_x)_{x \in D}$ be the corresponding family of their laws on the canonical path-space $C(\R_+; \R^d)$. Following standard definitions from the literature, see e.g.\ \cite[Chapter III]{revuzyor} and the notion of a Markov family of measures in \cite[Section 2.5]{karatzasshreve}, we say that \eqref{eq:generalSVIE} has the time-homogeneous Markov property, if for each Borel set $A \subseteq D$ and $t \geq 0$, the function $D \ni x \longmapsto \P_x[X_t \in A]$ is measurable, and for all $0 \leq t < T$ and $x \in D$ it holds that
\begin{equation}\label{eq: Markov timehomogeneous}
    \P_{x}[ X_T \in A \, | \, \mathcal{F}_t] = \P_{X_t}[ X_{T-t} \in A], \qquad \P_{x}\text{-a.s.},
\end{equation} 
where $(\mathcal{F}_t)_{t \in \R_+}$ denotes the natural filtration generated by the coordinate process $X$. By approximation arguments, \eqref{eq: Markov timehomogeneous} clearly extends to $\E_x[f(X_T) \, | \, \mathcal{F}_t] = \E_{X_t}[ f(X_{T-t})]$ whenever $f$ is measurable with $\mathbb{E}_x\big[|f(X_r)|\big]<\infty$ for all $r\in\R_+$ and $x\in D$. Here, $\E_x$ denotes the expectation with respect to $\P_x$ on the canonical path-space. Note that $p_t(x,A):=\P_x[X_t \in A]$, $t \in \R_+$, defines stochastic transition kernels on $D\times \mathcal{B}(D)$ that satisfy by \eqref{eq: Markov timehomogeneous} the Chapman-Kolmogorov consistency condition $p_{u}(x, \cdot) = \int_D p_{u-t}(y, \cdot)\, p_{t}(x,\,\mathrm{d}y)$ for all $x\in D$ and $0\le t<u$. Thus, under assumption \eqref{eq: Markov timehomogeneous}, for each $x_0\in D$, the solution to \eqref{eq:generalSVIE} with path measure $\P_{x_0}$ is a $D$-valued Markov process according to \cite[Definition III.1.3]{revuzyor} with initial law $\delta_{x_0}$ and transition kernels $(p_t)_{t\in\R_+^*}$. 

In this work, we examine the failure of \eqref{eq: Markov timehomogeneous} for SVEs using two complementary approaches. Firstly, let us observe that the time-homogeneous Markov property \eqref{eq: Markov timehomogeneous} necessarily implies that the first moment $m_1(t;x) = \E_x[X_t]$ satisfies
\begin{align}\label{eq: 8}
    m_1(t+s;x) = \mathbb{E}_x\big[m_1(s;X_t)\big], \qquad s,t\in\R_+,\  x\in D,
\end{align}
which follows from the semigroup property of the Markov semigroup $P_t\varphi(x):=\E_x[\varphi(X_t)]$. Hence, to explicitly compute the first moment of the process, in Section 2 we suppose that the drift is affine linear, i.e. $b(x) = b + \beta x$ with $b \in \R^d$ and $\beta \in \R^{d\times d}$, and that the initial curve satisfies $g(t; x) = g_0(t) + g_1(t)x$ where $g_0(0) = 0$ and $g_1(0) = \mathrm{id}_{\R^d}$. Under these assumptions, we show that all solutions of \eqref{eq: 8} are characterised by a relation of the Volterra kernel $K_b$ and the corresponding functions $g_0,g_1$. As a consequence, for any given Volterra kernel $K_b$, we may always find $g_0,g_1$ such that the corresponding Volterra process satisfies condition \eqref{eq: 8}. Conversely, if $g_0, g_1$ are not of the prescribed form determined by $K_b$, then \eqref{eq: 8} does not hold, and $X$ cannot be a time-homogeneous Markov process.

Subsequently, we proceed to study the analogue for the second moment $m_2(t;x) = \mathbb{E}_x[ X_t \otimes X_t]\in \R^{d\times d}$, i.e.\ we characterise all solutions of
\[
    m_2(t+s;x) = \mathbb{E}_x\big[ m_2(s;X_t) \big], \qquad s,t \in \R_+, \ x \in D.
\]
For this purpose, we further assume that \eqref{eq:generalSVIE} is an affine Volterra process with $K:=K_b=K_{\sigma}$, which allows us to compute $m_2$ semi-explicitly in terms of resolvents. Combined with \eqref{eq: 8}, this additional condition turns out to be highly restrictive and it can only be satisfied if the Volterra kernel $K$ is, up to time-dependent orthogonal rotations, a matrix exponential. For the Volterra square-root process, we show that such orthogonal rotations are necessarily trivial, while for the Volterra Ornstein-Uhlenbeck process, they are fully unconstrained. The latter provides a novel class of Gaussian processes that are not Markov, but whose time-marginals satisfy the Chapman-Kolmogorov equations and hence coincide with those of some Markov process. 

In our second approach, we study the failure of the time-homogeneous Markov property for~\eqref{eq:generalSVIE} with general drifts, $K = K_b = K_{\sigma}$, and under the technical assumption $d=m=1$ for the dimension. Our method is based on a reduction of the problem to Gaussian processes via the small-time CLT derived in \cite{FGW24} and a generalization to time-dependent initial curves $g$ as given in Theorem \ref{theorem: CLT}. More specifically, we show that, if~\eqref{eq:generalSVIE} were Markov, then the Gaussian process
\begin{align}\label{eq: Y Gaussian}
Y_t = \int_0^t \overline{K}(t-s)\,\mathrm{d}B_s, \quad t\in\R_+,
\end{align}
would also be Markov, with $\overline{K}$ related to $K$ via
\[
    \lim_{n \to \infty} \left( \int_0^{1/n} K(r)^2\, \mathrm{d}r \right)^{-1} \int_0^{s/n} K\left( \frac{t-s}{n} + r\right)K(r)\, \mathrm{d}r = \int_0^s \overline{K}(t-s + r)\overline{K}(r)\, \mathrm{d}r,
\]
for all $0<s\le t$, cf.\ \cite[Cond.~(ii) in Thm.~2.3]{FGW24}. However, the Markov property for centered Gaussian processes can be characterized explicitly in terms of their autocovariance (see \cite[Theorem V.8.1]{doob1953stochastic}, \cite[Lemma~2.8]{BaGe26}).
In this work, we directly prove that \eqref{eq: Y Gaussian} fails to be Markov for the fractional kernel $\overline{K}(t) = t^{H - 1/2}$ with $H\in\R_+^*\setminus\{\frac{1}{2}\}$, in line with the framework of self-similar Gaussian processes from \cite{BaGe26}. Remark that our second method applies to $b,\sigma$ that are merely H\"older continuous without further structural restrictions.

\subsubsection*{Notation}
Throughout the paper, we use the notations $\mathbb{R}_+ := [0,\infty)$, $\mathbb{R}_+^* := (0,\infty)$, and $\mathbb{R}^* := \mathbb{R}\setminus\{0\}$. Here and below, we write
\[
    (f \ast g)(t) = \int_0^t f(t-s)g(s)\, \mathrm{d}s
\]
where $t\in\R_+$ and $f,g$ are matrix- or vector-valued functions such that $f(t-s) g(s)$ is well-defined, and the integral exists.

\subsubsection*{Structure of the work}

In Section \ref{subsection:newmomentformula}, we present our computational approach for the class of SVEs with affine linear drift and initial curve. In Section \ref{subsection:RLnonMarkov}, we provide a direct proof that \eqref{eq: Y Gaussian} is not a Markov process for the fractional Riemann-Liouville kernel. Finally, our last section is dedicated to the alternative approach via small-time CLTs that covers a large class of SVEs \eqref{eq:generalSVIE} with general Hölder continuous coefficients, and under a
mild additional assumption on the small-time behavior of the kernel.

\section{Moment formula approach}\label{subsection:newmomentformula}

\subsection{Characterisation of the Markov property for the first moment}

In this section, we study the failure of the time-homogeneous Markov property for the particular case of~\eqref{eq:generalSVIE} where the drift is an affine function of the process, i.e. weak solutions of the $D$-valued SVE
\begin{align}\label{eq:generalaffineSVIE}
     X_t &= g(t; x)+\int_{0}^{t}K_b(t-s) \hspace{0.03cm}(b+\beta X_{s})\, \mathrm{d}s + \int_0^t K_{\sigma}(t-s)\hspace{0.03cm}\sigma(X_s)\, \mathrm{d}B_s,\quad t \in \R_{+},
\end{align} 
where $g(\cdot; x):\R_+ \longrightarrow \R^d$, $b \in \R^d$, $\beta \in \R^{d\times d}$, $\sigma: \R^d \longrightarrow \R^{d \times m}$ denotes the diffusion coefficient, $K_b: \R_+ \longrightarrow \R^{d \times d}$, and $K_{\sigma}: \R_+ \longrightarrow \R^{d \times d}$ are measurable Volterra kernels, and $B$ is an $m$-dimensional standard Brownian motion. The choice of affine drifts allows us to explicitly compute the first moment of the process in terms of resolvents.  

For the Markov property, the value $X_0 = g(0;x)=x$ plays a central role. To identify this value and perform explicit computations, we shall assume that the function $g(t) = g(t;x)$ is an affine function of~$x\in D$, i.e.
\begin{align}\label{eq: g affine}
    g(t;x) = g_0(t) + g_1(t)x,
\end{align}
where $g_0: \R_+ \longrightarrow \R^d$ and $g_1: \R_+ \longrightarrow \R^{d \times d}$. Here and below, we shall assume that the following conditions are satisfied:
\begin{assumption}\label{assumption first moment MP}
    The state-space $D \subseteq \R^d$ has a non-empty interior. The diffusion coefficient $\sigma: \R^d \longrightarrow \R^{d \times m}$ is measurable with linear growth. The Volterra kernels satisfy $K_b \in L_{\mathrm{loc}}^1(\R_+; \R^{d \times d})$, and $K_{\sigma} \in L_{\mathrm{loc}}^2(\R_+; \R^{d \times d})$. Moreover, $g_0: \R_+ \longrightarrow \R^d$ and $g_1: \R_+ \longrightarrow \R^{d \times d}$ are continuous on $\R_+$ and satisfy $g_0(0) = 0$ and $g_1(0) = \mathrm{id}_{\R^d}$.
\end{assumption}

Since $\sigma$ has linear growth and $g_0, g_1$ are locally bounded, one can show that any solution $X$ of~\eqref{eq:generalaffineSVIE} satisfies $\sup_{t \in [0,T]}\E_x[|X_t|^p] < \infty$ for each $p, T > 0$ and $x\in D$, see e.g.\ \cite[Theorem~1.4]{weak_solution}. In particular, by taking expectations in \eqref{eq:generalaffineSVIE}, we find that its first moment solves the linear Volterra equation
\[
    \E_x[X_T] = g_0(T) + g_1(T)x + \int_0^T K_b(T-t)\left( b+ \beta\hspace{0.02cm} \E_x[X_t]\right)\, \mathrm{d}t.
\]
The latter can be solved explicitly in terms of resolvents, see \cite[Lemma 4.2]{AbiJaLaPu19}. Let $R_b \in L_{\mathrm{loc}}^1(\R_+; \R^{d \times d})$ be the resolvent of the second kind for the kernel $-K_b\beta$, i.e.\ the unique solution of
\begin{align}\label{eq: Rb}
    R_b + (-K_b\beta) \ast R_b = R_b + R_b \ast (-K_b\beta) = -K_b\beta. 
\end{align}
By the general Volterra theory of \cite{grippenberg}, equations of the form $f = h + K_b \beta \ast f$ have the unique solution $f = h - R_b \ast h$. In particular, we find that  
\begin{align}\label{eq: first moment}
    \E_x[X_t] = \mathcal{E}_0(t) + \mathcal{E}_1(t)x, \quad t\in\R_+,
\end{align}
where $\mathcal{E}_0 \in C(\R_+; \R^d)$, $\mathcal{E}_1 \in C(\R_+; \R^{d \times d})$ are defined by
\begin{align*}
    \mathcal{E}_0(t) &= g_0(t) + \int_0^t E_b(t-s)\beta g_0(s)\, \mathrm{d}s + \int_0^t E_b(s)b\, \mathrm{d}s,
    \\ \mathcal{E}_1(t) &= g_1(t) + \int_0^t E_{b}(t-s)\beta g_1(s)\, \mathrm{d}s,
\end{align*}
with $E_b \in L_{\mathrm{loc}}^1(\R_+; \R^{d\times d})$ given by $E_b = K_b - R_b \ast K_b$, thus satisfying $R_b = -E_b\beta$. Inserting this into the definition of $E_{b}$, we see that it is the unique solution of the resolvent equation
\begin{align}\label{eq: Eb resolvent equation}
    E_b = K_b + E_b\beta \ast K_b.
\end{align}
The following is our first main result. It provides a characterisation of the \textit{Markov property for the first moment} of the process, as specified in \eqref{eq: MP first moment}.

\begin{theorem}\label{thm:affinetimehom}
    Suppose that Assumption \ref{assumption first moment MP} is satisfied, and for each $x \in D$, equation \eqref{eq:generalaffineSVIE} has a weak solution $(X, \mathcal{F}, (\mathcal{F}_t)_{t\in\R_+}, \P_x)$. Then the following are equivalent:
    \begin{enumerate}
        \item[(a)] For each $x \in D$ and all $0 \leq t < T$ one has
        \begin{align}\label{eq: MP first moment}
        \E_x[ X_T ] = \E_x[ \E_{X_t}[X_{T-t}]].
        \end{align}
        \item[(b)] There exist $A \in \R^{d \times d}$ and $B \in \R^d$ such that for $t\in\R_+$:
        \[
            \mathcal{E}_0(t) = \left(\int_0^t \mathrm{e}^{sA}\, \mathrm{d}s \right)B \quad \text{ and } \quad \mathcal{E}_1(t) = \mathrm{e}^{tA}.
        \]
        \item[(c)] There exist $A \in \R^{d \times d}$ and $B \in \R^d$ such that the functions $g_0, g_1$ are given by
        \begin{align*}
            g_0(t) &= \int_0^t \mathrm{e}^{sA}B\,\mathrm{d}s - \int_0^t K_b(t-s)\beta \left( \int_0^s \mathrm{e}^{rA}B\,\mathrm{d}r \right) \mathrm{d}s - \int_0^t K_b(s)b\,\mathrm{d}s,
            \\ g_1(t) &= \mathrm{e}^{tA} - \int_0^t K_b(t-s)\beta \mathrm{e}^{sA}\,\mathrm{d}s.
        \end{align*}
        \item[(d)] $g_0, g_1$ are absolutely continuous, and there exist $A \in \R^{d \times d}$, $B \in \R^d$ such that 
        \[
            K_b(t)\beta = g_1(t)A - g_1'(t) \quad \text{ and } \quad K_b(t)b = g_1(t)B - g_0'(t).
        \]
    \end{enumerate}
\end{theorem}
\begin{proof}
    $(a) \Longleftrightarrow (b):$ By direct computation and using \eqref{eq: first moment}, we find for all $x \in D$ and $0 \leq t < T$:
    \begin{align*}
        \E_x[ \E_{X_t}[X_{T-t}]] &= \E_x\left[ \mathcal{E}_0(T-t) + \mathcal{E}_1(T-t)X_t \right]
        \\ &= \mathcal{E}_0(T-t) + \mathcal{E}_1(T-t)\mathcal{E}_0(t) + \mathcal{E}_1(T-t)\mathcal{E}_1(t)x.
    \end{align*}
    Hence, property (a) is equivalent to 
    \begin{align}\label{eq: 3}
        \mathcal{E}_0(T) + \mathcal{E}_1(T)x = \mathcal{E}_0(T-t) + \mathcal{E}_1(T-t)\mathcal{E}_0(t) + \mathcal{E}_1(T-t)\mathcal{E}_1(t)x.
    \end{align}
    This relation holds under condition (b). For the converse implication, note that \eqref{eq: 3} is equivalent to $\left(\mathcal{E}_1(T) - \mathcal{E}_1(T-t)\mathcal{E}_1(t)\right)x = \mathcal{E}_0(T-t) + \mathcal{E}_1(T-t)\mathcal{E}_0(t) - \mathcal{E}_0(T)$. Evaluating this for $x,y \in D$, and then taking their difference yields
    \begin{align}\label{eq: 4}
        \big(\mathcal{E}_1(T) - \mathcal{E}_1(T-t)\hspace{0.02cm}\mathcal{E}_1(t)\big)\hspace{0.02cm}(x-y) = 0.
    \end{align}
    Since $D$ has non-empty interior, $D - D = \{ x-y \ : \ x,y \in D\}$ contains a neighbourhood of the origin. Hence, $\mathrm{span}(D-D) = \R^d$. In particular, \eqref{eq: 4} implies that $\mathcal{E}_1$ satisfies Cauchy's multiplicative functional equation $\mathcal{E}_1(T) = \mathcal{E}_1(T-t)\hspace{0.02cm}\mathcal{E}_1(t)$ for $0 \leq t < T$. Since $\mathcal{E}_1$ is continuous with $\mathcal{E}_1(0) = \mathrm{id}_{\R^d}$, we find a unique $A \in \R^{d\times d}$ such that $\mathcal{E}_1(t) = \mathrm{e}^{tA}$ for $t \geq 0$, see e.g.\ \cite[Theorem~2.9]{EngelNagel}. Likewise, inserting this back into \eqref{eq: 3}, we also obtain that
    \begin{align}\label{eq: E0}
        \mathcal{E}_0(T) = \mathcal{E}_0(T-t) + \mathrm{e}^{(T-t) A}\mathcal{E}_0(t), \qquad 0 \leq t < T.
    \end{align}
    Integration over $t \in [0,T]$ gives $T \mathcal{E}_0(T) = \int_0^T \mathcal{E}_0(t)\, \mathrm{d}t + \int_0^T \mathrm{e}^{(T-t)A}\mathcal{E}_0(t)\, \mathrm{d}t$. This shows that $\mathcal{E}_0$ is also continuously differentiable on $(0,\infty)$. Evaluating \eqref{eq: E0} at $T = t+s$ with $s \geq 0$ gives $\mathcal{E}_0(t+s) = \mathcal{E}_0(s) + \mathrm{e}^{sA}\mathcal{E}_0(t)$, and hence 
    \[
        \frac{\mathcal{E}_0(t+s) - \mathcal{E}_0(s)}{t} = \mathrm{e}^{sA} \frac{\mathcal{E}_0(t)}{t}.
    \]
    Taking the limit $t \searrow 0$ shows by $\mathcal{E}_0(0) = 0$ that $\mathcal{E}_0'(0)$ exists and satisfies $\mathcal{E}_0'(s) = \mathrm{e}^{sA}\mathcal{E}_0'(0)$. Therefore, integration yields $\mathcal{E}_0(t) = \left(\int_0^t \mathrm{e}^{s A}\, \mathrm{d}s \right) \mathcal{E}_0'(0)$. Setting $B = \mathcal{E}_0'(0)$ completes the proof.

    $(b) \Longleftrightarrow (c):$ Using the representations for $\mathcal{E}_0, \mathcal{E}_1$, we find that $g_0, g_1$ satisfy the linear Volterra equations
    \begin{align}\label{eq: 5}
        \begin{cases} 
        g_0(t) + \int_0^t E_{b}(t-s)\beta g_0(s)\,\mathrm{d}s = \int_0^t \left( \mathrm{e}^{sA}B - E_b(s)b\right)\, \mathrm{d}s,
        \\ 
        g_1(t) + \int_0^t E_b(t-s)\beta g_1(s)\,\mathrm{d}s = \mathrm{e}^{tA}. 
        \end{cases}
    \end{align}
    Again by the general Volterra theory of \cite{grippenberg}, equations of the form $f + (E_b \beta) \ast f = h$ have the unique solution $f = h - \widetilde{R} \ast h$, where $\widetilde{R}$ is the resolvent of the second kind uniquely determined by $\widetilde{R} + (E_b \beta) \ast \widetilde{R} = E_b \beta = \widetilde{R} + \widetilde{R} \ast (E_b\beta)$. Hence, $E_b\beta=-R_b$ combined with \eqref{eq: Rb} and the uniqueness of resolvents implies that $\widetilde{R} = K_b\beta$. Applying this to $g_1$ with $h = \mathrm{e}^{\cdot A}$ yields
    \begin{align*}
        g_1(t) &= \mathrm{e}^{tA} - \int_0^t K_b(t-s)\hspace{0.01cm}\beta \hspace{0.02cm}\mathrm{e}^{sA}\,\mathrm{d}s.
    \end{align*}
    Similarly, for $g_0$ we find
    \begin{align*}
        g_0(t) &= \int_0^t \mathrm{e}^{sA}B\,\mathrm{d}s - \int_0^t E_b(s)b\,\mathrm{d}s 
        \\ &\quad - \int_0^t K_b(t-s)\beta \left( \int_0^s \mathrm{e}^{rA}B\,\mathrm{d}r - \int_0^s E_b(r)b\,\mathrm{d}r \right) \mathrm{d}s
        \\ &= \int_0^t \mathrm{e}^{sA}B\,\mathrm{d}s - \int_0^t K_b(t-s)\beta \left( \int_0^s \mathrm{e}^{rA}B\,\mathrm{d}r \right) \mathrm{d}s - \int_0^t K_b(s)b\,\mathrm{d}s,
    \end{align*}
    where the second identity follows from the associativity of the convolution operator, the resolvent equation \eqref{eq: Rb},
    \begin{align*}
        -E_b + (K_b \beta) \ast E_b &= - K_b + R_b \ast K_b + (K_b \beta) \ast K_b - (K_b \beta) \ast R_b \ast K_b
        \\ &= -K_b + (K_b \beta) \ast K_b + \left( R_b + (-K_b\beta) \ast R_b\right)\ast K_b
        = -K_b,
    \end{align*}
    and $\mathrm{id}_{\R^d}\ast ((K_b \beta) \ast E_b)=(K_b \beta)\ast ( \mathrm{id}_{\R^d}\ast E_b)$. This proves the desired equivalence.
    
    $(c) \Longleftrightarrow (d):$ The explicit formulas given for $g_0, g_1$ imply that these functions are absolutely continuous. Differentiating $g_1$ gives
    \begin{align*}
        g_1'(t) &= \mathrm{e}^{tA}A - K_b(t)\beta - \int_0^t K_b(s)\beta \mathrm{e}^{(t-s)A} A\, \mathrm{d}s
        \\ &= \left( \mathrm{e}^{tA} - \int_0^t K_b(s)\beta \mathrm{e}^{(t-s)A}\, \mathrm{d}s\right)A - K_b(t) \beta 
        = g_1(t)A - K_b(t) \beta,
    \end{align*}
    and hence $K_b(t)\beta = g_1(t)A - g_1'(t)$. Likewise, differentiating $g_0$ gives
    \begin{align*}
        g_0'(t) &= \mathrm{e}^{tA}B - \int_0^t K_b(s)\hspace{0.01cm}\beta \hspace{0.02cm}\mathrm{e}^{(t-s)A}B\, \mathrm{d}s - K_b(t)b
        = g_1(t)B - K_b(t)b.
    \end{align*}
    This proves $(c) \implies (d)$. For the converse implication, it suffices to solve the differential equation for $g_1$ via the variation of constants formula and to subsequently integrate the differential equation for~$g_0$.
\end{proof}

The additional degrees of freedom $A \in \R^{d \times d}$ and $B \in \R^d$ appearing in (b) parameterise all solutions of \eqref{eq: MP first moment} in terms of the first moment \eqref{eq: first moment}. Likewise, condition (c) parameterises $g_0,g_1$ in terms of $K_b$ and $A,B$, while (d) allows us to recover $K_b$ from $g_1$ when $\beta$ is invertible. Finally, under assumption~(d), noting that $g_1$ is continuous with $g_1(0) = \mathrm{id}_{\R^d}$, and that the space of invertible matrices is open, we find $t_0 > 0$ small enough such that $g_1(t)$ is invertible for $t \in [0,t_0)$. Hence, solving the equations from~(d) with respect to $A,B$ gives $A = g_1(t)^{-1}\big( K_b(t)\beta + g_1'(t)\big)$ and $B = g_1(t)^{-1}\big( K_b(t)b + g_0'(t)\big)$.

\begin{example}[exponential-type kernel]
     Suppose that $K_b(t) = \sum_{j=1}^N c_j e^{-\lambda_j t}\mathrm{id}_{\R^d}$ is scalar-valued with $c_j \neq 0$ and $\lambda_j \in \R$ pairwise distinct. Then for any choice of $A \in \R^{d \times d}$ and $B \in \R^d$ the Markov property of the first moment \eqref{eq: MP first moment} is satisfied for
        \begin{align*}
            g_1(t) &= \mathrm{e}^{tA} - \beta \sum_{j=1}^N c_j \hspace{0.02cm}\mathrm{e}^{-\lambda_j t} \int_0^t \mathrm{e}^{r (A + \lambda_j)}\, \mathrm{d}r,
            \\ g_0'(t) &= \mathrm{e}^{tA}B - \sum_{j=1}^N c_j \hspace{0.02cm}\mathrm{e}^{-\lambda_j t}\left( \beta \int_0^t \mathrm{e}^{r(A + \lambda_j)}\, \mathrm{d}r B + b \right).
        \end{align*}
\end{example}

\begin{example}[fractional Riemann-Liouville kernel]
    Suppose that $K_b$ is scalar-valued and given by $K_b(t) = \frac{t^{\alpha-1}}{\Gamma(\alpha)}\hspace{0.02cm}\mathrm{id}_{\R^d}$ with $\alpha \in (0,1)$. Then for any choice of $A \in \R^{d \times d}$ and $B \in \R^d$, the Markov property of the first moment \eqref{eq: MP first moment} is satisfied for
    \begin{align*}
        g_1(t) &= \mathrm{e}^{tA} - \frac{1}{\Gamma(\alpha)} \int_0^t s^{\alpha-1} \beta \mathrm{e}^{(t-s)A} \,\mathrm{d}s, \\
        g_0(t) &= \left( \int_0^t g_1(s) \,\mathrm{d}s \right) B - \frac{t^\alpha}{\Gamma(\alpha+1)} b.
    \end{align*}
\end{example}

Both examples show that, for a given $K_b$, we may find the corresponding $g_0, g_1$ such that the Markov property of the first moment \eqref{eq: MP first moment} holds. Conversely, if $g_0, g_1$ are not of the form given above, then \eqref{eq: MP first moment} fails, and the process cannot possess the time-homogeneous Markov property. Below we complement such examples by showing that condition~(d) forces a particular form for the Volterra kernel~$K_b$, but it does not uniquely determine it, i.e.\ for fixed $A, B$, there exist in general several admissible choices for $g_0, g_1$, and vice versa.

\begin{example}
    Suppose that $g_0 \equiv 0$, $\beta$ is invertible and $b\neq 0$. Then according to property (d), \eqref{eq: MP first moment} is equivalent to the existence of $A \in \R^{d \times d}$ and $B \in \R^d$ such that $K_b(t)\beta = g_1(t)A - g_1'(t)$ and $K_b(t)b = g_1(t)B$ hold. Hence, it follows 
    \begin{align}\label{eq: 6}
        K_b(t) = (g_1(t)A - g_1'(t))\beta^{-1},
    \end{align}
    which determines $K_b$ in terms of $g_1$. Inserting this into the second equation in (d) gives $\big( g_1(t)A - g_1'(t)\big)\beta^{-1}b = g_1(t)B$. Rearranging the terms yields
    \[
        g_1'(t) \left[ \beta^{-1}b \right] = g_1(t) \left[ A\beta^{-1}b - B \right].
    \]
    This provides a structural restriction on the parameters $A,B$ and the function $g_1$ such that \eqref{eq: MP first moment} is satisfied. This relation can be viewed from two perspectives:
    \begin{enumerate}
        \item[(i)] Let us choose $A = (1-\lambda)\mathrm{id}_{\R^d}$ with $\lambda \in \R$ and $B = \beta^{-1}b$, which corresponds by condition (b) to the first moment
        \[
            \E_x[X_T] = \frac{\mathrm{e}^{(1-\lambda)T} - 1}{1-\lambda}\beta^{-1}b + \mathrm{e}^{(1-\lambda)T}x,
        \]
        when $\lambda\neq 1$. Then $g_1'(t)[\beta^{-1}b] = -\lambda g_1(t)[\beta^{-1}b]$, which implies 
        $g_1(t)[\beta^{-1}b] = \mathrm{e}^{-\lambda t}\beta^{-1}b$. Consequently, \eqref{eq: 6} gives $K_b(t)b = \mathrm{e}^{-\lambda t}\beta^{-1}b$. However, when $d > 1$, we still have the freedom to choose the action of $g_1$ orthogonal to the vector $\beta^{-1}b$. Hence, $K_b$ is not uniquely determined by~\eqref{eq: 6}.
        
        \item[(ii)] Suppose that $g_1(t) = h_1(t)\mathrm{id}_{\R^d}$ with a scalar-valued function $h_1$ with $h_1(0) = 1$. Then we obtain
        \[
            h_1'(t) \beta^{-1} b = h_1(t)\left( A \beta^{-1}b - B\right)
        \]
        and hence there must exist $\lambda \in \R$ such that $-\lambda \beta^{-1} b = A \beta^{-1}b - B$. If $\lambda = 0$, then $B = A \beta^{-1}b$ and $h_1'(t) = 0$, which gives $h_1\equiv 1$. Inserting this into \eqref{eq: 6} yields $K_b(t) = A\beta^{-1}$. If $\lambda \neq 0$, then $h_1'(t)[\beta^{-1} b] = -\lambda h_1(t) [\beta^{-1} b]$, which, by $\beta^{-1}b\neq 0$, implies $h_1(t) = \mathrm{e}^{-\lambda t}$. Hence, $K_b(t) = (A+\lambda \mathrm{id}_{\R^d})\beta^{-1}\mathrm{e}^{-\lambda t}$. In both cases, the matrix $A$ is, in general, not uniquely determined, which provides additional degrees of freedom for $K_b$ to satisfy \eqref{eq: MP first moment}.
    \end{enumerate}
\end{example}

Our results apply in particular to affine and polynomial Volterra processes: 
\begin{enumerate}
    \item[(i)] The Volterra Ornstein-Uhlenbeck process on $D = \R^d$ with $\sigma(x) = \sigma_0 \in \R^{d \times d}$, where existence and uniqueness of solutions is established in \cite[Section 5]{AbiJaLaPu19}.
    \item[(ii)] The Volterra square-root process on $D = \R_+^d$ with $b \in \R_+^d$, $\beta_{ij} \geq 0$ for $i \neq j$, and $\sigma(x) = \mathrm{diag}(\sigma_1 \sqrt{x_1}, \dots, \sigma_d \sqrt{x_d})$ with $\sigma_1,\dots, \sigma_d \geq 0$, see \cite[Theorem~6.1]{AbiJaLaPu19}. This clearly covers \eqref{eq:roughHeston}.
    \item[(iii)] The Jacobi Volterra process on $D = [\alpha_1, \alpha_2]$ with $\alpha_1 < \alpha_2$, and $\sigma(x) = \sqrt{(x-\alpha_1)(\alpha_2-x)}$, for which weak existence and uniqueness is established in \cite[Corollary 2.8]{10.1214/24-EJP1234}. 
\end{enumerate}

Besides these examples, Theorem~\ref{thm:affinetimehom} applies to a broad class of stochastic Volterra equations with affine drifts and a diffusion coefficient $\sigma$ that is continuous and of linear growth, provided that weak existence of solutions can be established.

\subsection{Characterisation of the Markov property for the second moment}

Theorem \ref{thm:affinetimehom} demonstrates that the Markov property for the first moment \eqref{eq: MP first moment} is, in general, not sufficient to force the Volterra kernel into an exponential form. Indeed, for a given Volterra kernel, we may always choose $g_0,g_1$ according to the formulas provided in part (c). In this section, we demonstrate that adding a similar requirement onto the second moment removes this freedom and, under minor regularity conditions, implies that the Volterra kernel is a matrix exponential. To compute the second moment explicitly, we shall further assume that $x \longmapsto \sigma \sigma^{\top}(x)$ is affine, and $K := K_b = K_{\sigma}$. This leads to affine Volterra processes on the state space $D$ defined by
\begin{align}\label{eq: AVP}
    X_{t} &= g_0(t) + g_1(t)x + \int_{0}^{t}K(t-s)\left(b + \beta X_{s}\right)\, \mathrm{d}s + \int_0^t K(t-s)\, \mathrm{d}M_s^X, \quad t\in\R_+, 
\end{align}
where $x \in D$, $K \in L_{\mathrm{loc}}^2(\R_+; \R^{d \times d})$, and $M^X$ is a continuous martingale with quadratic variation
\[
    \langle M^X \rangle_t = \int_0^t \left(\alpha_0 + \sum_{i = 1}^d X_{s,i} \hspace{0.02cm}\alpha_i \right)\, \mathrm{d}s
\]
with symmetric and positive semi-definite $\alpha_0, \alpha_1, \dots, \alpha_d \in \R^{d \times d}$. Remark that \eqref{eq: AVP} is a continuous affine Volterra process in the sense of \cite{AbiJaLaPu19} with the minor modification that the constant initial condition is replaced by $g(t;x)=g_0(t) + g_1(t)x$. It contains the Volterra square-root process, the Volterra Heston model, and the Volterra Ornstein-Uhlenbeck process as particular cases.

Since $g_0,g_1$ are continuous, it follows that $\sup_{t \in [0,T]}\E_x[|X_t|^p] < \infty$ for $p, T > 0$, see \cite{weak_solution}. Its first moment is given by \eqref{eq: first moment}. To compute its second moment, let us introduce the resolvent kernel $E = K - R \ast K \in L_{\mathrm{loc}}^2(\R_+; \R^{d \times d})$, which satisfies \eqref{eq: Eb resolvent equation}, and $R$ is defined as in \eqref{eq: Rb} for $K_b=K=K_{\sigma}$. Then $\E_x[X_t] = \mathcal{E}_0(t) + \mathcal{E}_1(t)x$ is given by \eqref{eq: first moment}, and $X$ satisfies by the Volterra variation of constants formula (see \cite[Lemma 2.5]{AbiJaLaPu19}): 
\[
    X_t = \E_x[X_t] + \int_0^t E(t-s)\, \mathrm{d}M_s^X.
\]
Let us write $v \otimes v = v v^{\top}$ with components $(v \otimes v)_{ij} = v_i v_j$ for $v \in \R^d$. Remark that, if $M \in \R^{d \times d}$, then 
\[
    M (v \otimes v)M^{\top} = (Mv) \otimes (Mv).
\]
With this notation, applying the It\^{o} isometry to the stochastic integral yields for the second moment of the process 
\begin{align*}
    \E_x[ X_t\otimes X_t ] &= \E_x[X_t] \otimes \E_x[X_t] + \mathrm{var}(X_t)
    \\ &= \left( \mathcal{E}_0(t) + \mathcal{E}_1(t)x \right)^{\otimes 2} + \E_x\left[ \int_0^t E(t-s) \,\mathrm{d}\langle M^X\rangle_s\hspace{0.02cm} E(t-s)^{\top} \right]
    \\ &= \mathcal{E}_0(t) \otimes \mathcal{E}_0(t) + \mathcal{E}_0(t) \otimes (\mathcal{E}_1(t)x) + ( \mathcal{E}_1(t)x)\otimes \mathcal{E}_0(t) + \mathcal{E}_1(t) (x\otimes x) \mathcal{E}_1(t)^{\top}
    \\ &\qquad + \int_0^t E(t-s) \left( \alpha_0 + \sum_{i=1}^d \E_x[X_{s,i}]\hspace{0.02cm} \alpha_i \right)E(t-s)^{\top}\, \mathrm{d}s.
\end{align*}
Substituting the explicit form of the first moment $\E_x[X_{s,i}] = \mathcal{E}_{0,i}(s) + (\mathcal{E}_1(s)x)_i$ as given in \eqref{eq: first moment}, yields
\begin{align}\label{eq: second moment}
    \E_x[ X_t \otimes X_t ] &= \mathcal{V}_0(t) + \mathcal{V}_1(t,x) + \mathcal{E}_1(t) (x \otimes x) \mathcal{E}_1(t)^{\top},
\end{align}
where $\mathcal{V}_0$ is independent of $x$ and $\mathcal{V}_1$ is linear in $x$, given by
\begin{align*}
    \mathcal{V}_0(t) &= \mathcal{E}_0(t) \otimes \mathcal{E}_0(t) + \int_0^t E(t-s) \left( \alpha_0 + \sum_{i=1}^d \mathcal{E}_{0,i}(s)\hspace{0.02cm}\alpha_i\right) E(t-s)^{\top}\, \mathrm{d}s,
\\ \mathcal{V}_1(t,x) &= \mathcal{E}_0(t) \otimes (\mathcal{E}_1(t)x) + (\mathcal{E}_1(t)x) \otimes \mathcal{E}_0(t) + \sum_{i=1}^d \int_0^t (\mathcal{E}_1(s)x)_i E(t-s) \hspace{0.02cm}\alpha_i\hspace{0.02cm} E(t-s)^{\top}\, \mathrm{d}s.
\end{align*}

The following is our second main result. It provides a characterisation of the Markov property for the first two moments of affine Volterra processes.

\begin{theorem}
    Suppose that the state-space $D$ has non-empty interior, that $g_0: \R_+ \longrightarrow \R^d$, $g_1: \R_+ \longrightarrow \R^{d\times d}$ are continuous, satisfy $g_0(0) = 0$, and $g_1(0) = \mathrm{id}_{\R^d}$, and $K \in L_{\mathrm{loc}}^2(\R_+; \R^{d \times d})$. Assume that for each $x \in D$, equation \eqref{eq: AVP} has a weak solution $(X, \mathcal{F}, (\mathcal{F}_t)_{t\in\R_+}, \P_x)$, and this collection satisfies \eqref{eq: MP first moment}. Then the following are equivalent:
    \begin{enumerate}
        \item[(a)] For all $0 \leq t \leq T$ and all $x \in D$ it holds
        \begin{align}\label{eq: MP second moment}
            \E_x[ X_T \otimes X_T ] = \E_x\big[ \E_{X_t}[ X_{T-t} \otimes X_{T-t} ] \big].
        \end{align}
        \item[(b)] For each $i \in \{0,\dots, d\}$ the following identity holds almost everywhere with respect to $t, \tau \geq 0$: 
            \begin{align}\label{eq: kernel structural relation}
                E(\tau + t) \alpha_i E(\tau + t)^{\top} = \mathcal{E}_1(\tau)E(t) \alpha_i E(t)^{\top}\mathcal{E}_1(\tau)^{\top}.
            \end{align}
    \end{enumerate}
\end{theorem}
\begin{proof}
Let us evaluate the right-hand side of \eqref{eq: MP second moment}. An application of \eqref{eq: second moment} almost surely yields 
\[
    \E_{X_t}[ X_{T-t} \otimes X_{T-t} ] = \mathcal{V}_0(T-t) + \mathcal{V}_1(T-t, X_t) + \mathcal{E}_1(T-t) (X_t \otimes X_t) \mathcal{E}_1(T-t)^{\top}.
\]
Taking the outer expectation $\E_x[\cdot]$ and using the linearity of $\mathcal{V}_1$ in $x$, yields
\begin{align} \label{eq: MP second moment functional}
    \E_x[ \E_{X_t}[ X_{T-t} \otimes X_{T-t} ] ] &= \mathcal{V}_0(T-t) + \mathcal{V}_1\big(T-t, \E_x[X_t]\big) + \mathcal{E}_1(T-t) \E_x[X_t \otimes X_t] \mathcal{E}_1(T-t)^{\top}.
\end{align}
By substituting $\E_x[X_t] = \mathcal{E}_0(t) + \mathcal{E}_1(t)x$ and \eqref{eq: second moment} into \eqref{eq: MP second moment functional}, and exploiting the linearity of $\mathcal{V}_1(t, \cdot)$, we obtain
\begin{align*}
    \E_x\big[ \E_{X_t}[ X_{T-t} \otimes X_{T-t} ] \big] 
    &= \mathcal{V}_0(T-t) + \mathcal{V}_1\big(T-t, \mathcal{E}_0(t) + \mathcal{E}_1(t)x\big) 
    \\ &\quad + \mathcal{E}_1(T-t) \Big( \mathcal{V}_0(t) + \mathcal{V}_1(t,x) + \mathcal{E}_1(t) (x \otimes x) \mathcal{E}_1(t)^{\top} \Big) \mathcal{E}_1(T-t)^{\top}
    \\ &= \mathcal{V}_0(T-t) + \mathcal{V}_1\big(T-t, \mathcal{E}_0(t)\big) + \mathcal{V}_1\big(T-t, \mathcal{E}_1(t)x\big)
    \\ &\quad + \mathcal{E}_1(T-t) \mathcal{V}_0(t) \mathcal{E}_1(T-t)^{\top} + \mathcal{E}_1(T-t) \mathcal{V}_1(t,x) \mathcal{E}_1(T-t)^{\top}
    \\ &\quad + \mathcal{E}_1(T-t) \mathcal{E}_1(t) (x \otimes x) \mathcal{E}_1(t)^{\top} \mathcal{E}_1(T-t)^{\top}.
\end{align*}
For \eqref{eq: MP second moment} to hold, this expression must equal $\E_x[ X_T \otimes X_T ] = \mathcal{V}_0(T) + \mathcal{V}_1(T, x) + \mathcal{E}_1(T) (x \otimes x) \mathcal{E}_1(T)^{\top}$ for all $x \in D$. Since $D$ has a non-empty interior, the polynomial coefficients of $x$ must coincide. Comparing the constant, linear, and quadratic terms in $x$, we obtain the following system of relations for all $0 \leq t \leq T$:

\textit{(i) The quadratic terms} yield the relation
\begin{align*}
    \mathcal{E}_1(T) (x \otimes x) \mathcal{E}_1(T)^{\top} = \mathcal{E}_1(T-t) \mathcal{E}_1(t) (x \otimes x) \mathcal{E}_1(t)^{\top} \mathcal{E}_1(T-t)^{\top},
\end{align*}
which is already satisfied by assumption \eqref{eq: MP first moment} and the characterisation of Theorem \ref{thm:affinetimehom}.
    
\textit{(ii) The linear terms} in $x$ yield
    \begin{align}\label{eq: MP linear}
        \mathcal{V}_1(T, x) = \mathcal{V}_1\big(T-t, \mathcal{E}_1(t)x\big) + \mathcal{E}_1(T-t) \mathcal{V}_1(t,x) \mathcal{E}_1(T-t)^{\top}.
    \end{align}
    To rewrite this relation, let us first note that the left-hand side equals 
    \begin{align*}
        \mathcal{V}_1(T, x) = \mathcal{E}_0(T) \otimes (\mathcal{E}_1(T)x) + (\mathcal{E}_1(T)x) \otimes \mathcal{E}_0(T) 
     + \sum_{i=1}^d \int_0^T (\mathcal{E}_1(s)x)_i E(T-s) \alpha_i E(T-s)^{\top}\, \mathrm{d}s.
    \end{align*}
    For the right-hand side of~\eqref{eq: MP linear}, we first expand $\mathcal{V}_1\big(T-t, \mathcal{E}_1(t)x\big)$. Using the semigroup property $\mathcal{E}_1(T-t)\mathcal{E}_1(t) = \mathcal{E}_1(T)$ from Theorem \ref{thm:affinetimehom}.(b), and applying the change of variables $u = s+t$ to the integral, we get
    \begin{align*}
    \mathcal{V}_1\big(T-t, \mathcal{E}_1(t)x\big) &= \mathcal{E}_0(T-t) \otimes (\mathcal{E}_1(T)x) + (\mathcal{E}_1(T)x) \otimes \mathcal{E}_0(T-t) 
    \\ &\quad + \sum_{i=1}^d \int_t^T (\mathcal{E}_1(u)x)_i E(T-u) \alpha_i E(T-u)^{\top}\, \mathrm{d}u.
    \end{align*}
    Expanding the second term on the right-hand side of~\eqref{eq: MP linear} yields
    \begin{align*}
        \mathcal{E}_1(T-t) \mathcal{V}_1(t, x) \mathcal{E}_1(T-t)^{\top} &= (\mathcal{E}_1(T-t)\mathcal{E}_0(t)) \otimes (\mathcal{E}_1(T)x) + (\mathcal{E}_1(T)x) \otimes \big(\mathcal{E}_1(T-t)\mathcal{E}_0(t)\big) 
        \\ &\quad + \sum_{i=1}^d \int_0^t (\mathcal{E}_1(s)x)_i \mathcal{E}_1(T-t) E(t-s) \alpha_i E(t-s)^{\top} \mathcal{E}_1(T-t)^{\top}\, \mathrm{d}s.
    \end{align*}
    Inserting both expressions into \eqref{eq: MP linear}, using the relation $\mathcal{E}_0(T) = \mathcal{E}_0(T-t) + \mathcal{E}_1(T-t)\mathcal{E}_0(t)$ provided by Theorem \ref{thm:affinetimehom} (see \eqref{eq: E0}), and finally cancelling the integral $\int_t^T (\mathcal{E}_1(u)x)_i E(T-u) \alpha_i E(T-u)^{\top}\, \mathrm{d}u$ which appears on both sides, we arrive at 
    \begin{align*}
        &\ \sum_{i=1}^d \int_0^t (\mathcal{E}_1(s)x)_i E(T-s) \alpha_i E(T-s)^{\top}\, \mathrm{d}s 
        \\ &\qquad = \sum_{i=1}^d \int_0^t (\mathcal{E}_1(s)x)_i \mathcal{E}_1(T-t) E(t-s) \alpha_i E(t-s)^{\top} \mathcal{E}_1(T-t)^{\top}\, \mathrm{d}s.
    \end{align*}
    Since this must hold for all $x \in D$ and the interior of $D$ is non-empty, we find after the substitution $T-t = \tau \geq 0$ for all $j \in \{1,\dots, d\}$ the relation $\sum_{i = 1}^d \int_0^t \mathcal{E}_{1, ij}(s) E(\tau + t - s)\alpha_i E(\tau + t - s)^{\top}\, \mathrm{d}s
        = \sum_{i=1}^d \int_0^t \mathcal{E}_{1,ij}(s) \mathcal{E}_1(\tau)E(t-s)\alpha_i E(t-s)^{\top}\mathcal{E}_1(\tau)^{\top}\, \mathrm{d}s$. Now define $F_i(u) = E(\tau + u)\alpha_i E(\tau + u)^{\top} - \mathcal{E}_1(\tau)E(u)\alpha_i E(u)^{\top}\mathcal{E}_1(\tau)^{\top}$. Then using $t-s = u$, we arrive at 
    \begin{align*}
        \sum_{i=1}^d \int_0^t \mathcal{E}_{1,ij}(t-u)F_i(u)\, \mathrm{d}u = 0, \qquad \forall j \in \{1,\dots, d\}.
    \end{align*}
    Since $\mathcal{E}_1(t) = \mathrm{e}^{tA}$ for some $A \in \R^{d \times d}$ by Theorem \ref{thm:affinetimehom}, we may differentiate this identity in $t$. Since $\mathcal{E}_{1,ij}(0) = \delta_{ij}$, this yields the linear Volterra equation $F(t) + \int_0^t \mathcal{E}'_{1}(t-u)^{\top}F(u)\, \mathrm{d}u = 0$. By uniqueness, we find $F \equiv 0$, which proves \eqref{eq: kernel structural relation} for $i \in \{1,\dots, d\}$.     
    
\textit{(iii) The constant terms} yield the relation
    \begin{align}\label{eq: MP constant}
        \mathcal{V}_0(T) = \mathcal{V}_0(T-t) + \mathcal{V}_1\big(T-t, \mathcal{E}_0(t)\big) + \mathcal{E}_1(T-t) \mathcal{V}_0(t) \mathcal{E}_1(T-t)^{\top}.
    \end{align}
    Using again the explicit form of $\mathcal{V}_0, \mathcal{V}_1$ yields 
    \begin{align*}
        &\ \mathcal{V}_0(T-t) + \mathcal{V}_1\big(T-t, \mathcal{E}_0(t)\big) + \mathcal{E}_1(T-t) \mathcal{V}_0(t) \mathcal{E}_1(T-t)^{\top}
        \\ &= \mathcal{E}_0(T-t) \otimes \mathcal{E}_0(T-t) + \mathcal{E}_0(T-t)\otimes \big(\mathcal{E}_1(T-t)\mathcal{E}_0(t)\big) + \big(\mathcal{E}_1(T-t)\mathcal{E}_0(t)\big) \otimes \mathcal{E}_0(T-t) 
        \\ &\qquad + \mathcal{E}_1(T-t) ( \mathcal{E}_0(t)\otimes \mathcal{E}_0(t))\mathcal{E}_1(T-t)^{\top}
        + \int_0^{T-t}E(T-t-s)\left( \alpha_0 + \sum_{i=1}^d \mathcal{E}_{0,i}(s)\alpha_i\right) E(T-t-s)^{\top}\, \mathrm{d}s 
        \\ &\qquad + \sum_{i=1}^d \int_{0}^{T-t}(\mathcal{E}_1(s)\mathcal{E}_0(t))_i E(T-t-s)\alpha_i E(T-t-s)^{\top}\, \mathrm{d}s
        \\ &\qquad + \int_0^t \mathcal{E}_1(T-t)E(t-s)\left( \alpha_0 + \sum_{i=1}^d \mathcal{E}_{0,i}(s)\alpha_i \right) E(t-s)^{\top}\mathcal{E}_1(T-t)^{\top}\, \mathrm{d}s,
    \end{align*}
    and similarly
    \[
        \mathcal{V}_0(T) = \mathcal{E}_0(T) \otimes \mathcal{E}_0(T) + \int_0^T E(T-s)\left( \alpha_0 + \sum_{i=1}^d \mathcal{E}_{0,i}(s)\alpha_i \right) E(T-s)^{\top}\, \mathrm{d}s.
    \]
    Noting that $\mathcal{E}_0(T) = \mathcal{E}_0(T-t) + \mathcal{E}_1(T-t)\mathcal{E}_0(t)$ by Theorem \ref{thm:affinetimehom}, see \eqref{eq: E0}, we find
    \begin{align*}
        \mathcal{E}_0(T) \otimes \mathcal{E}_0(T) &= \mathcal{E}_0(T-t) \otimes \mathcal{E}_0(T-t) + \mathcal{E}_0(T-t)\otimes \big(\mathcal{E}_1(T-t)\mathcal{E}_0(t)\big) + \big(\mathcal{E}_1(T-t)\mathcal{E}_0(t)\big) \otimes \mathcal{E}_0(T-t) 
        \\ &\qquad + \mathcal{E}_1(T-t) ( \mathcal{E}_0(t)\otimes \mathcal{E}_0(t))\mathcal{E}_1(T-t)^{\top}.
    \end{align*}
    Similarly, for the integrals we obtain
    \begin{align*}
       &\ \int_0^{T-t}E(T-t-s)\left( \alpha_0 + \sum_{i=1}^d \mathcal{E}_{0,i}(s)\alpha_i\right) E(T-t-s)^{\top}\, \mathrm{d}s 
        \\ &\qquad + \sum_{i=1}^d \int_{0}^{T-t}(\mathcal{E}_1(s)\mathcal{E}_0(t))_i E(T-t-s)\alpha_i E(T-t-s)^{\top}\, \mathrm{d}s
        \\ &= \int_0^{T-t}E(T-t-s)\alpha_0 E(T-t-s)^{\top}\, \mathrm{d}s + \sum_{i=1}^d \int_{0}^{T-t}(\mathcal{E}_{0}(s) + \mathcal{E}_1(s)\mathcal{E}_0(t))_i E(T-t-s)\alpha_i E(T-t-s)^{\top}\, \mathrm{d}s
        \\ &= \int_0^{T-t}E(T-t-s)\alpha_0 E(T-t-s)^{\top}\, \mathrm{d}s + \sum_{i=1}^d \int_{0}^{T-t}\mathcal{E}_{0,i}(t+s) E(T-t-s)\alpha_i E(T-t-s)^{\top}\, \mathrm{d}s
        \\ &= \int_0^{T-t}E(s)\alpha_0 E(s)^{\top}\, \mathrm{d}s + \sum_{i=1}^d \int_{t}^{T}\mathcal{E}_{0,i}(s) E(T-s)\alpha_i E(T-s)^{\top}\, \mathrm{d}s.
    \end{align*}
    Hence, relation \eqref{eq: MP constant} reduces to
    \begin{align*}
        &\ \int_{T-t}^T E(s) \alpha_0 E(s)^{\top}\, \mathrm{d}s + \sum_{i=1}^d \int_0^t \mathcal{E}_{0,i}(s) E(T-s)\alpha_i E(T-s)^{\top}\, \mathrm{d}s
        \\ &= \int_0^t \mathcal{E}_1(T-t)E(t-s) \alpha_0 E(t-s)^{\top}\mathcal{E}_1(T-t)^{\top}\, \mathrm{d}s 
        \\ &\qquad + \sum_{i=1}^d \int_0^t \mathcal{E}_{0,i}(s) \mathcal{E}_1(T-t)E(t-s) \alpha_i E(t-s)^{\top}\mathcal{E}_1(T-t)^{\top}\, \mathrm{d}s.
    \end{align*}
    Considering the above equation for $T = \tau + t$ with $\tau \geq 0$ and then substituting $t-s = u$, we arrive at the identity
    \begin{align*}
        &\int_{0}^t E(\tau + u) \left( \alpha_0 + \sum_{i=1}^d \mathcal{E}_{0,i}(t-u) \alpha_i \right) E(\tau + u)^{\top}\, \mathrm{d}u  
        \\ &= \int_0^t \mathcal{E}_1(\tau)E(u) \left( \alpha_0 + \sum_{i=1}^d \mathcal{E}_{0,i}(t-u)\alpha_i\right) E(u)^{\top}\mathcal{E}_1(\tau)^{\top}\, \mathrm{d}u. 
    \end{align*}
    Using \eqref{eq: kernel structural relation} for $i=1,\dots, d$ yields $\int_0^t E(\tau + u)\alpha_0 E(\tau + u)^{\top}\, \mathrm{d}u = \int_0^t \mathcal{E}_1(\tau)E(u)\alpha_0 E(u)^{\top}\mathcal{E}_1(\tau)^{\top}\, \mathrm{d}u$, and hence proves \eqref{eq: kernel structural relation} for $i = 0$ by differentiation. The converse implication $(b) \Longrightarrow (a)$ follows from the same identities.
\end{proof}

Let $O(d)$ be the orthogonal group of $d\times d$-matrices $O$ that satisfy $O O^{\top} = \mathrm{id}_{\R^d}$. While \eqref{eq: MP first moment} does, in general, not suffice to force the Volterra kernel into an exponential form, the addition of the second moment condition \eqref{eq: MP second moment} turns out to be sufficient under mild regularity conditions.

\begin{corollary}\label{cor: discrete centralizer}
Suppose that the state-space $D$ has non-empty interior, that $g_0: \R_+ \longrightarrow \R^d$ and $g_1: \R_+ \longrightarrow \R^{d\times d}$ are continuous, satisfy $g_0(0) = 0$, and $g_1(0) = \mathrm{id}_{\R^d}$. Assume that for each $x \in D$, equation \eqref{eq: AVP} has a weak solution $(X, \mathcal{F}, (\mathcal{F}_t)_{t\in\R_+}, \P_x)$, and this collection satisfies \eqref{eq: MP first moment} and \eqref{eq: MP second moment}, i.e.\ the Markov property for the first and second moments. Assume that $\Gamma := \sum_{i=0}^d \alpha_i$ is positive definite, that $K \in C((0,\infty); \R^{d \times d}) \cap L_{\mathrm{loc}}^2(\R_+; \R^{d \times d})$, and the resolvent kernel $E$ is invertible on $\R_+^*$. 

Then there exists a continuous family of orthogonal matrices $(O_t)_{t > 0}$ in $\R^{d \times d}$ and a symmetric positive definite matrix $F$ such that
\begin{align}\label{eq: E representation t positive}
    E(t) = \mathcal{E}_1(t) F^{1/2} O_t \Gamma^{-1/2}, \qquad t > 0.
\end{align}
Furthermore, for each $i \in \{0, \dots, d\}$, the matrix $O_t \widetilde{\alpha}_i O_t^{\top}$ is constant in $t > 0$, where $\widetilde{\alpha}_i = \Gamma^{-1/2}\alpha_i \Gamma^{-1/2}$. Moreover, if the centralizer of $\{ \widetilde{\alpha}_0, \dots, \widetilde{\alpha}_d \}$ in the orthogonal group given by
\[
    \mathcal{C} = \big\{ O \in O(d) \ : \ O\widetilde{\alpha}_i = \widetilde{\alpha}_i O \text{ for all } i\in \{0, \dots, d\} \big\}
\]
is discrete, then $E$ can be continuously extended to $t = 0$ with $E(0)$ being invertible, and the resolvent kernel reduces to the pure matrix exponential form
\[
    E(t) = \mathcal{E}_1(t)E(0), \qquad t \geq 0.
\]
Consequently, the original Volterra kernel $K$ is also a matrix exponential.
\end{corollary}
\begin{proof}
Since $\Gamma$ is symmetric and positive definite, there exists a unique symmetric positive definite square root $\Gamma^{1/2}$. Taking the sum over $i \in \{0, \dots, d\}$ in \eqref{eq: kernel structural relation} and using $T = \tau + t$ yields 
\begin{align}\label{eq: 10}
    E(T) \Gamma E(T)^{\top} = \mathcal{E}_1(T-t)E(t)\Gamma E(t)^{\top} \mathcal{E}_1(T-t)^{\top}.
\end{align}
By noting that $\mathcal{E}_1(t)^{-1} = \mathcal{E}_1(-t)$ due to its exponential form provided by Theorem \ref{thm:affinetimehom}, rearranging terms yields $\mathcal{E}_1(-T)E(T)\Gamma E(T)^{\top} \mathcal{E}_1(-T)^{\top} = \mathcal{E}_1(-t)E(t) \Gamma E(t)^{\top} \mathcal{E}_1(-t)^{\top}$. Since this holds for all $0 < t \leq T$, both sides must be equal to a constant matrix $F \in \R^{d \times d}$. Because $\Gamma$ is positive definite and $E(t)$ is invertible, $F$ is symmetric and positive definite. Thus, we find $E(t) \Gamma E(t)^{\top} = \mathcal{E}_1(t)F \mathcal{E}_1(t)^{\top}$ for $t > 0$, which can be rewritten as
\[
    (E(t)\Gamma^{1/2})(E(t)\Gamma^{1/2})^{\top} = (\mathcal{E}_1(t)F^{1/2})(\mathcal{E}_1(t)F^{1/2})^{\top}.
\]
It is a standard result in linear algebra that if $AA^{\top} = BB^{\top}$ and $B$ is invertible, then $A = BO$ for some orthogonal matrix $O$. Therefore, there exists a family of orthogonal matrices $(O_t)_{t > 0}$ such that $E(t)\Gamma^{1/2} = \mathcal{E}_1(t)F^{1/2}O_t$, which yields $E(t) = \mathcal{E}_1(t)F^{1/2}O_t \Gamma^{-1/2}$. Because $K$ is locally square integrable, so is $R$ given by \eqref{eq: Rb} and hence also $E$. Thus $R \ast K \in C(\R_+; \R^{d \times d})$, and hence $E = K - R \ast K$ is continuous on $(0,\infty)$. Likewise, since $g_0, g_1$ are continuous, also $\mathcal{E}_1$ is continuous due to the explicit form given after \eqref{eq: first moment}. Thus, the map $t \longmapsto O_t = F^{-1/2}\mathcal{E}_1(-t)E(t)\Gamma^{1/2}$ is continuous. 

Substituting this form of $E(t)$ back into \eqref{eq: kernel structural relation} and using the semigroup property $\mathcal{E}_1(\tau+t) = \mathcal{E}_1(\tau)\mathcal{E}_1(t)$, we obtain
\[
    \mathcal{E}_1(\tau+t)F^{1/2} O_{\tau+t} \widetilde{\alpha}_i O_{\tau+t}^{\top} F^{1/2}\mathcal{E}_1(\tau+t)^{\top} = \mathcal{E}_1(\tau+t)F^{1/2} O_{t} \widetilde{\alpha}_i O_{t}^{\top} F^{1/2}\mathcal{E}_1(\tau+t)^{\top},
\]
where $\widetilde{\alpha}_i = \Gamma^{-1/2}\alpha_i\Gamma^{-1/2}$. Canceling the invertible matrices on both sides gives $O_{\tau+t} \widetilde{\alpha}_i O_{\tau+t}^{\top} = O_t \widetilde{\alpha}_i O_t^{\top}$ and completes the proof of the first assertion. 

To prove the second assertion, note that for fixed $s > 0$, $(O_s^{\top}O_t)\widetilde{\alpha}_i (O_s^{\top}O_t)^{\top} = \widetilde{\alpha}_i$ holds for all $t \geq s$. Hence, the continuous family of orthogonal matrices $[s, \infty) \ni t \longmapsto O_s^{\top}O_t$ lies in the centralizer $\mathcal{C}$ of the set $\{ \widetilde{\alpha}_0, \dots, \widetilde{\alpha}_d \}$. Assuming that $\mathcal{C}$ is discrete, any continuous map taking values in $\mathcal{C}$ must be constant. Since $O_s^{\top}O_s = \mathrm{id}_{\R^d}$, we must have $O_s^{\top}O_t = \mathrm{id}_{\R^d}$ for all $t \geq s$, and hence $O_t = O_s \equiv O$ for all $0 < s \leq t$. Taking the limit $t \searrow 0$ in \eqref{eq: E representation t positive}, we find 
\[
    \lim_{t \searrow 0} E(t) = \lim_{t \searrow 0} \mathcal{E}_1(t) F^{1/2} O \Gamma^{-1/2} = F^{1/2} O \Gamma^{-1/2} =: E(0).
\]
Replacing $F^{1/2} O \Gamma^{-1/2}$ with $E(0)$ in \eqref{eq: E representation t positive}, we conclude that $E(t) = \mathcal{E}_1(t)E(0)$ for all $t \geq 0$. 

To conclude that the original Volterra kernel $K$ is also a matrix exponential, recall from Theorem \ref{thm:affinetimehom} that $\mathcal{E}_1(t) = \mathrm{e}^{tA}$ for some matrix $A \in \R^{d \times d}$. Thus, the resolvent kernel takes the form $E(t) = \mathrm{e}^{tA}E(0)$ and satisfies the differential equation $E'(t) = A E(t)$. Furthermore, by \eqref{eq: Eb resolvent equation}, the Volterra kernel $K$ is linked to $E$ via the Volterra resolvent equation $K(t) = E(t) - \int_0^t E(t-s) \beta K(s)\, \mathrm{d}s$. Since $E$ is differentiable, $K$ is also differentiable. Applying the Leibniz rule to differentiate this integral equation with respect to $t$ yields
\begin{align*}
    K'(t) &= E'(t) - E(0)\beta K(t) - \int_0^t E'(t-s)\beta K(s)\, \mathrm{d}s \\
    &= A E(t) - E(0)\beta K(t) - A \int_0^t E(t-s)\beta K(s)\, \mathrm{d}s \\
    &= A \left( E(t) - \int_0^t E(t-s)\beta K(s)\, \mathrm{d}s \right) - E(0)\beta K(t) = (A - E(0)\beta) K(t).
\end{align*}
With the initial condition $K(0) = E(0)$, this homogeneous linear ordinary differential equation admits the unique solution $K(t) = \mathrm{e}^{t(A - E(0)\beta)}E(0)$. Thus, $K$ is a pure matrix exponential, which completes the proof.
\end{proof}

We close this section with a discussion of the most important examples of affine Volterra processes~\eqref{eq: AVP}. For the general theory and additional details on their construction, we refer to~\cite{AbiJaLaPu19}.

\begin{example}[one-dimensional affine Volterra processes]
    Suppose that $d = 1$, $K \in C((0,\infty)) \cap L_{\mathrm{loc}}^2(\R_+)$, and assume that the corresponding resolvent $E$ does not vanish on $(0,\infty)$. Let $X$ be a one-dimensional affine Volterra process of the form \eqref{eq: AVP} with nontrivial diffusion (i.e. $\alpha_0 + \alpha_1 > 0$) satisfying \eqref{eq: MP first moment} and \eqref{eq: MP second moment}. Since the orthogonal group in dimension one is discrete, given by reflections $O(1) = \{-1, 1\}$, also $\mathcal{C} \subset O(1)$ is discrete. By Corollary \ref{cor: discrete centralizer}, there exist $A, B \in \R$ such that the Volterra kernel is of the form $K(t) = \mathrm{e}^{tA}B$, $t \geq 0$.
\end{example}

The assumption $E \neq 0$ on $(0,\infty)$ is satisfied if $K$ is additionally completely monotone and not identically zero. Indeed, any completely monotone function that is not identically zero is strictly positive. If $\beta \geq 0$, then $E \geq K > 0$ by complete monotonicity, while for $\beta < 0$, the resolvent $E$ itself is by the relation $R = -E\beta$ completely monotone, not identically zero, and hence strictly positive.

\begin{example}[affine Volterra processes on $D = \R_+^d$]
    Suppose that $D = \R_+^d$, and let $X$ be the Volterra square-root process given by
    \[
        X_t = g_0(t) + g_1(t)x + \int_0^t K(t-s)\hspace{0.02cm}(b + \beta X_s)\, \mathrm{d}s + \int_0^t K(t-s)\hspace{0.02cm}\sigma(X_s)\, \mathrm{d}B_s,\quad t\in\R_+,
    \]
    where $\sigma(x) = \mathrm{diag}(\sqrt{c_1 x_1}, \dots, \sqrt{c_d x_d})$ with $c_1,\dots, c_d > 0$, and $K = \mathrm{diag}(K_1,\dots, K_d)$ is such that $0 < K_1,\dots, K_d \in L_{\mathrm{loc}}^2(\R_+)$ are completely monotone, see \cite[Section~6]{AbiJaLaPu19} for the study of such processes. Then $X$ is an affine Volterra process with $\alpha_0 = 0$ and $\alpha_i = c_i e_i e_i^{\top}$ for $i \in \{1,\dots, d\}$, where $e_i$ is the $i$-th standard basis vector. In particular, $\Gamma = \sum_{i=1}^d \alpha_i = \mathrm{diag}(c_1, \dots, c_d)$ is strictly positive definite. The scaled matrices are given by $\widetilde{\alpha}_i = \Gamma^{-1/2}\alpha_i \Gamma^{-1/2} = e_i e_i^{\top}$. Since any orthogonal matrix $O$ that commutes with all $e_i e_i^{\top}$ is necessarily diagonal and hence has entries $\pm 1$, it follows that the centralizer $\mathcal{C}$ is a finite group isomorphic to $\{-1, 1\}^d$. Hence, if the resolvent $E$ is invertible on $(0,\infty)$, then Corollary \ref{cor: discrete centralizer} implies that $K(t) = \mathrm{e}^{tM}E(0)$ for some diagonal matrix $M$, as soon as the Markov property is satisfied for the first and second moments, i.e.\ \eqref{eq: MP first moment} and~\eqref{eq: MP second moment}.
\end{example}

Likewise, if $K$ is scalar-valued (i.e. $K_1 = \dots = K_d$), not identically zero, and $\beta = \beta^{\top}$ is negative definite, then $E$ is completely monotone and hence invertible for all $t > 0$. This covers the most important case of the Volterra square-root process with mean-reversion.

\begin{example}[Multivariate Volterra Ornstein-Uhlenbeck process]
    Consider a $d$-dimensional Volterra Ornstein-Uhlenbeck process on $D = \R^d$ with $d \geq 2$ given by
    \begin{align}\label{eq: VOU}
        X_t = g_0(t) + g_1(t)x + \int_0^t K(t-s)\left( b+\beta X_s\right)\, \mathrm{d}s + \int_0^t K(t-s)\Sigma^{1/2} \, \mathrm{d}B_s,
    \end{align}
    where $K \in C((0,\infty); \R^{d \times d}) \cap L_{\mathrm{loc}}^2(\R_+; \R^{d \times d})$ is such that the resolvent $E$ is invertible on $(0,\infty)$, and $\Sigma \in \R^{d \times d}$ is symmetric and positive definite. Then $X$ is an affine Volterra process with $\alpha_0 = \Sigma$ and $\alpha_1 = \dots = \alpha_d = 0$. The scaled matrices are given by $\widetilde{\alpha}_0 = \Sigma^{-1/2}\Sigma \Sigma^{-1/2} = \mathrm{id}_{\R^d}$ and $\widetilde{\alpha}_i = 0$ for $i \geq 1$. Hence, the centralizer given by $\mathcal{C} = O(d)$ is not discrete when $d \geq 2$. The first part of Corollary \ref{cor: discrete centralizer} gives the representation
    \begin{align}\label{eq: K rotational}
        E(t) = \mathcal{E}_1(t)F^{1/2}O_t \Sigma^{-1/2}, \qquad t > 0,
    \end{align}
    for some symmetric positive semi-definite matrix $F$, and a continuous family $(O_t)_{t > 0} \subset O(d)$. However, the orthogonal family $(O_t)_{t > 0}$ does not need to be given by a matrix-exponential. For instance, for any skew-symmetric matrix $Y \in \R^{d \times d}$, the choice $O_t = \mathrm{e}^{\psi(t)Y}$ with continuous $\psi: (0,\infty) \longrightarrow \R$ yields $E(t) = O_t \Sigma^{-1/2}$ with $K$ implicitly defined by $K = E - E\beta \ast K$, such that $X$ given by \eqref{eq: VOU} satisfies \eqref{eq: MP first moment} and \eqref{eq: MP second moment}. 
\end{example}

In this final example, one might hope that imposing the Markov property on higher-order moments, analogously to \eqref{eq: MP first moment} and \eqref{eq: MP second moment}, would uniquely determine $E$, and hence $K$, to be of matrix-exponential form. The following proposition shows, however, that this is not the case. In fact, \eqref{eq: K rotational} is equivalent to the Chapman–Kolmogorov equations.

\begin{proposition}
    Let $X$ be the $d$-dimensional Volterra Ornstein-Uhlenbeck process given by~\eqref{eq: VOU} with state-space $D = \R^d$, $\Sigma \in \R^{d \times d}$ symmetric and positive definite, $b \in \R^d$, and $\beta \in \R^{d \times d}$. Suppose that $g_0: \R_+ \longrightarrow \R^d$ and $g_1: \R_+ \longrightarrow \R^{d \times d}$ are continuous and satisfy $g_0(0) = 0$ and $g_1(0) = \mathrm{id}_{\R^d}$. Assume that $K \in C((0,\infty); \R^{d \times d}) \cap L_{\mathrm{loc}}^2(\R_+; \R^{d \times d})$, and that the resolvent $E$ is invertible on $(0,\infty)$. Then the following are equivalent:
    \begin{enumerate}
        \item[(a)] The one-dimensional time-marginals $p_t(x, \cdot) = \P_x[X_t \in \cdot]$ with $x \in \R^d$ and $t \geq 0$ satisfy the Chapman-Kolmogorov equations, i.e.
        \[
            p_{t+s}(x, \cdot) = \int_{\R^d} p_s(y,\cdot)\, p_t(x,\mathrm{d}y), \qquad s,t \geq 0, \ x \in \R^d.
        \]
        \item[(b)] There exist a matrix $A \in \R^{d \times d}$, a vector $B \in \R^d$, and a family of orthogonal matrices $(O_t)_{t > 0} \subset O(d)$ and a symmetric positive-definite matrix $F$ such that for all $t > 0$:
        \begin{align*}
            \mathcal{E}_1(t) = \mathrm{e}^{tA}, \quad \mathcal{E}_0(t) = \left( \int_0^t \mathrm{e}^{sA}\, \mathrm{d}s \right)B, \quad E(t) = \mathrm{e}^{tA}F^{1/2}O_t \Sigma^{-1/2}.
        \end{align*}
    \end{enumerate}   
\end{proposition}
\begin{proof}
    By the variation of constants formula for stochastic Volterra equations, equation \eqref{eq: VOU} can be rewritten to $X_t = m(t,x) + \int_0^t E(t-s)\Sigma^{1/2}\, \mathrm{d}B_s$, where $m(t,x):=\E_x[ X_t] = \mathcal{E}_0(t) + \mathcal{E}_1(t)x$ is given as in \eqref{eq: first moment}. As a Gaussian process, its distribution is uniquely determined by its mean $m(t,x)$ and its covariance function
    \begin{equation}\label{eq: VOU variance general}
        V(t) = \int_0^t E(s)\Sigma E(s)^{\top}\, \mathrm{d}s.
    \end{equation}
    By approximation, property (a) has an equivalent form expressed in terms of characteristic functions
    \begin{equation}\label{eq: VOU Chapman Kolmogorov general}
        \E_x\big[ \mathrm{e}^{\mathrm{i} u^{\top} X_{t+\tau}}\big] = \E_x\Big[ \E_{X_t}\big[ \mathrm{e}^{ \mathrm{i} u^{\top} X_{\tau}}\big] \Big], \qquad t,\tau \geq 0, \ x,u \in \R^d.
    \end{equation}
    The left-hand side of \eqref{eq: VOU Chapman Kolmogorov general} evaluates to
    \begin{equation}\label{eq: LHS char function general}
        \E_x\big[ \mathrm{e}^{\mathrm{i} u^{\top} X_{t+\tau}}\big] = \exp\left( \mathrm{i} u^{\top} m(t+\tau, x) - \frac{1}{2} u^{\top} V(t+\tau) u \right),
    \end{equation}
    while the right-hand side takes the form
    \begin{align*}
        \E_x\Big[ \E_{X_t}\big[ \mathrm{e}^{\mathrm{i} u^{\top} X_{\tau}}\big] \Big]
        &= \E_x\left[ \mathrm{e}^{\mathrm{i} u^{\top} m(\tau, X_t) - \frac{1}{2}u^{\top}V(\tau)u} \right]
        \\ &= \mathrm{e}^{\mathrm{i} u^{\top}\mathcal{E}_0(\tau) - \frac{1}{2}u^{\top}V(\tau)u}\, \E_x\left[ \mathrm{e}^{\mathrm{i}(\mathcal{E}_1(\tau)^{\top}u)^{\top}X_t} \right]
        \\ &= \exp\left( \mathrm{i} u^{\top} \big[ \mathcal{E}_0(\tau) + \mathcal{E}_1(\tau)m(t,x) \big] - \frac{1}{2} u^{\top} \big[ V(\tau) + \mathcal{E}_1(\tau)V(t)\mathcal{E}_1(\tau)^{\top} \big] u \right).
    \end{align*}
    Since \eqref{eq: VOU Chapman Kolmogorov general} is supposed to hold for all $x, u \in \R^d$, we can equate the mean vectors and covariance matrices, which yields the set of conditions
    \begin{align}
        m(t+\tau, x) &= \mathcal{E}_0(\tau) + \mathcal{E}_1(\tau)m(t,x), \label{eq: VOU mean match} \\
        V(t+\tau) &= V(\tau) + \mathcal{E}_1(\tau)V(t)\mathcal{E}_1(\tau)^{\top}. \label{eq: VOU variance match}
    \end{align}
    Substituting $m(t,x) = \mathcal{E}_0(t) + \mathcal{E}_1(t)x$ into \eqref{eq: VOU mean match} and separating the constant from the terms proportional to $x$, gives $\mathcal{E}_1(t+\tau) = \mathcal{E}_1(\tau)\mathcal{E}_1(t)$ and $\mathcal{E}_0(t+\tau) = \mathcal{E}_0(\tau) + \mathcal{E}_1(\tau)\mathcal{E}_0(t)$. Since $g_0,g_1$ are continuous, the particular form of $\mathcal{E}_0, \mathcal{E}_1$ given after \eqref{eq: first moment} shows that $\mathcal{E}_0, \mathcal{E}_1$ are continuous. Hence, $\mathcal{E}_1(t) = \mathrm{e}^{tA}$ for some $A \in \R^{d \times d}$, and arguing similarly to \eqref{eq: E0} yields $\mathcal{E}_0(t) = \left( \int_0^t \mathrm{e}^{sA}\, \mathrm{d}s\right)B$ for some $B \in \R^d$.

    Finally, substituting the matrix exponential $\mathcal{E}_1(\tau) = \mathrm{e}^{\tau A}$ into condition \eqref{eq: VOU variance match} yields
    \[
        V(t+\tau) - V(\tau) = \mathrm{e}^{\tau A}V(t)\mathrm{e}^{\tau A^{\top}}.
    \]
    Dividing both sides by $t$ and taking the limit as $t \downarrow 0$ reveals that the right-derivative of $V$ at zero, denoted $F := \lim_{t \downarrow 0} \frac{1}{t}V(t)$, exists and is symmetric and positive definite. This yields $V'(\tau) = \mathrm{e}^{\tau A} F \mathrm{e}^{\tau A^{\top}}$ for $\tau > 0$. However, in view of \eqref{eq: VOU variance general}, for almost every $t > 0$ we also find $V'(t) = E(t)\Sigma E(t)^{\top}$. Since the right-hand side is continuous in $t$, this identity extends to all $t > 0$. Thus we obtain $E(t)\Sigma E(t)^{\top} = \mathrm{e}^{t A} F \mathrm{e}^{t A^{\top}}$ for $t > 0$. Since $F$ is symmetric and positive definite, there exists a unique symmetric positive definite square root $F^{1/2}$. This yields  
    \[
        (E(t)\Sigma^{1/2})(E(t)\Sigma^{1/2})^{\top} = (\mathrm{e}^{tA}F^{1/2})(\mathrm{e}^{tA}F^{1/2})^{\top}.
    \]
    Because $\Sigma^{1/2}$ is strictly positive definite, there exists a unique family of orthogonal matrices $(O_t)_{t > 0} \subset O(d)$ such that $E(t)\Sigma^{1/2} = \mathrm{e}^{tA}F^{1/2}O_t$, which yields assertion (b).

    Conversely, if (b) holds, substituting the explicit forms of $\mathcal{E}_0(t)$, $\mathcal{E}_1(t)$, and $E(t)$ into the definitions of the mean $m(t,x)$ and variance $V(t)$ directly verifies both \eqref{eq: VOU mean match} and \eqref{eq: VOU variance match}. This establishes the equivalence of the characteristic functions in \eqref{eq: VOU Chapman Kolmogorov general}, proving that the one-dimensional time-marginals satisfy the Chapman-Kolmogorov equations.
\end{proof}

Let us stress that the Chapman-Kolmogorov equations in~(a) alone are not sufficient to assert that~$X$ is a Markov process. However, by the general theory of Markov processes, we may find another Markov process $\widetilde{X}$ that has the same one-dimensional time marginals as~$X$. Thus, $K$ derived from \eqref{eq: K rotational} inserted into \eqref{eq: VOU} provides candidates of Gaussian processes that are not Markov, but whose time-marginals satisfy the Chapman-Kolmogorov equations. The failure of the Markov property could be shown, e.g., via the characterisation in terms of their covariance function, see e.g. \cite[Theorem V.8.1]{doob1953stochastic}.

\section{Gaussian Riemann-Liouville case}\label{subsection:RLnonMarkov}

As a preparation for Section \ref{subsection: NonMarkovviaCLT}, we provide in this section a new and direct proof that the fractional Riemann-Liouville process given by
\[
    Y_t := \int_{0}^{t}(t - s)^{H-1/2}\,\mathrm{d}B_{s}, \qquad t \in\R_+,
\]
is not a Markov process for $H \in\R_+^* \setminus \{\frac{1}{2}\}$. Firstly, observe that $Y$ is a continuous centered Gaussian process, whence its law is uniquely determined by its covariance function $c(t,s) = \E[ Y_t\hspace{0.02cm} Y_s ]$ with $s,t\in\R_+$. A criterion due to \cite[Theorem V.8.1]{doob1953stochastic} states that a centered Gaussian process has the (not necessarily time-homogeneous) Markov property
\begin{align}\label{eq: Markov property}
    \P[ Y_T \in \cdot \, | \, \F_t] = \P[ Y_T \in \cdot \, | \, Y_t  ],\hspace{0.3cm}\mbox{a.s.},
\end{align}
if and only if its covariance function satisfies for all $0 \leq s \leq t \leq u$ the identity
\begin{equation}\label{eq:GaussiannonMarkovfuncEq}
    c(s, u) \hspace{0.03cm}c(t,t) = c(s,t) \hspace{0.03cm}c(t,u).
\end{equation}
This condition can be effectively used to show that, e.g., fractional Brownian motion $B^H$ with Hurst parameter $H \in (0,1)$ is a Markov process only for $H = 1/2$, see \cite[Theorem 2.3]{nourdinFBM}. Further applications of this condition in the framework of self-similar Gaussian processes are given in \cite{BaGe26}. In particular, in \cite[Example 6.3]{BaGe26} it is shown that the Riemann-Liouville process $Y$ defined above does not possess the Markov property. Below, we give a novel and more direct proof based on regular conditional distributions. 

\begin{lemma}\label{lemma:RLnonMarkov}
   For every (Hurst) parameter $H\in\R_+^*\setminus \{\frac{1}{2}\}$, there are time points $\beta_1$, $\beta_2$ and $\beta_3$ satisfying $0<\beta_1<\beta_2<\beta_3$ and a finite interval
   $I\subset \mathbb R$ such that
   \begin{equation*}
     \mathbb{P}[Y_{\beta_3}\in I\, |\, Y_{\beta_2}=0, Y_{\beta_1}=1]
     >\mathbb{P}[Y_{\beta_3}\in I\, |\,  Y_{\beta_2}=0].
   \end{equation*}
\end{lemma}
This observation already rules out the Markov property: If $Y$ were a Markov process, it would satisfy
\begin{equation}\label{eq:Y Markov}
    \mathbb{P}[Y_{\beta_3}\in I \mid Y_{\beta_2}, Y_{\beta_1}] = \mathbb{P}[Y_{\beta_3}\in I \mid Y_{\beta_2}] \quad \text{a.s.}
\end{equation}
Because $Y$ is a non-degenerate Gaussian process, the joint law of $(Y_{\beta_1}, Y_{\beta_2})$ has full support on $\mathbb{R}^2$. The probabilities in the equation above are given by regular conditional distributions, which are continuous functions of the conditioning variables (see \cite[p.~65]{Gl04}). Therefore, the strict pointwise inequality at $(Y_{\beta_2}, Y_{\beta_1}) = (0, 1)$ established in Lemma \ref{lemma:RLnonMarkov} extends to an open neighborhood of this point. Since this neighborhood has strictly positive probability mass, the almost sure equality~\eqref{eq:Y Markov} fails.
\begin{proof}[Proof of Lemma~\ref{lemma:RLnonMarkov}]
  Let us choose $(\beta_3,\beta_2,\beta_1) = (\tau,\tau^2,\tau^3)$ for some suitable $\tau\in(0,1)$ determined below. As~$Y$ is a centered Gaussian process, the conditional law $(Y_{\tau}\in \cdot \, |\,  Y_{\tau^2}=0)$ is also centered. We now prove that, for sufficiently small $\tau>0$, the conditional Gaussian law $(Y_{\tau}\in \cdot \, |\,  Y_{\tau^2}=0, Y_{\tau^3}=1)$ has a non-zero expectation, which implies the assertion. Indeed, for any two nondegenerate Gaussian distributions on $\R$ with distinct parameters, there exists an interval $I$ on which the density of the first is strictly larger than that of the second.
  
  For notational convenience, define $\theta_{ij}=\theta_{ij}(\tau):=\mathbb{E}[ Y_{\tau^i}Y_{\tau^j}]$ for $i,j\in\{1,2,3\}$. By the classical expression for the conditional
  Gaussian mean, we obtain
  \begin{align}
      \mathbb{E}[Y_\tau \,|\, Y_{\tau^2}=0, Y_{\tau^3}=1]
      &=(\theta_{12},\theta_{13})
      \left(
        \begin{matrix}
            \theta_{22} & \theta_{23} \\
            \theta_{32} & \theta_{33} 
        \end{matrix}
      \right)^{-1}
       \left(
        \begin{matrix}
            0\\
            1
        \end{matrix}
      \right) \notag \\
      &= \frac{\theta_{13}\theta_{22}-\theta_{12}\theta_{23}}{\theta_{22}\theta_{33}-\theta_{23}\theta_{32}}.\label{eq:cond frac}
  \end{align}
  In the following, we are going to use the well-known formulas
  \begin{equation}\label{eq:RLCovviaHypergeometric}
    \int_0^s \big((t-u)(s-u)\big)^{H-1/2}\,\mathrm{d}u
     =\frac{t^{H-1/2}s^{H+1/2}}{H+\tfrac12}\,
     {}_2F_1\bigg( \genfrac{}{}{0pt}{}{\tfrac12-H,1}{H+\tfrac32} \,\bigg|\, \frac{s}{t}\bigg),
\end{equation}
  for $s\leq t$, which can be derived from Euler's integral representation of the Gaussian hypergeometric function ${}_2F_1$ (see \cite[Theorem 2.2.1]{Andrews_Askey_Roy_1999}) and Pfaff's transformation formula (see \cite[Eq.~(2.2.6)]{Andrews_Askey_Roy_1999}), as done in e.g.\ \cite[Eq.~(13)]{StefanroughBergomi}, and
  \[
     {}_2F_1\bigg(  \genfrac{}{}{0pt}{}{\tfrac12-H,1}{H+\tfrac32} \,\bigg|\, 1\bigg)
     =\frac{\Gamma(H+\tfrac32)\hspace{0.02cm}\Gamma(2H)}{\Gamma(2H+1)\hspace{0.02cm}\Gamma(H+\tfrac12)}
     =\frac{H+\tfrac12}{2H},
  \]
  which follows from Gauss's summation theorem for ${}_2F_1$ (see \cite[Theorem~2.2.2]{Andrews_Askey_Roy_1999}). Thus, as $\tau\searrow 0$, we can conclude by combining the above with ${}_2F_1(\dots|\, x)\sim 1$ as $x\searrow 0$: 
  \begin{align*}
  \theta_{kk}(\tau)&\sim\frac{\tau^{2kH}}{2H}\hspace{0.2cm} \mbox{for}\ k\in\{1,2,3\},\quad \theta_{k,k+1}(\tau)\sim\frac{\tau^{(2k+1)H+1/2}}{H+\tfrac12}\hspace{0.2cm} \mbox{for}\ k\in\{1,2\},\\ 
  \mbox{and}\hspace{0.2cm} \theta_{13}(\tau)&\sim \frac{\tau^{4H+1}}{H+\tfrac12}.
  \end{align*}
  This easily implies $\theta_{23}\theta_{32} \ll\theta_{22}\theta_{33}$, and, therefore, we can compute the asymptotics
  of~\eqref{eq:cond frac}, which yields
  \begin{align}\label{eq:RLcondExpAsymptoticsdelta0}
      \frac{\theta_{13}\theta_{22}-\theta_{12}\theta_{23}}{\theta_{22}\theta_{33}-\theta_{23}\theta_{32}}
      &\sim
      \frac{\Big(\frac{1}{(H+1/2)(2H)}-\frac{1}{(H+1/2)^2}\Big)\tau^{8H+1}}{\tau^{10H}/(2H)^2}
      =\frac{4H(1-2H)}{(2H+1)^2}\,\tau^{1-2H}.
  \end{align}
  As $H\neq\tfrac12$, this is non-zero for sufficiently small~$\tau$, which completes the proof.
 \end{proof}
 
Since $Y$ reduces for $H=1/2$ to classical Brownian motion, the prototypical example for a Gaussian Markov process in continuous time, one cannot expect the above argument to work. Indeed, the right-hand side of \eqref{eq:RLcondExpAsymptoticsdelta0} becomes $0$ in this case.

The analogue of Lemma \ref{lemma:RLnonMarkov} also holds if we condition on $Y_{\beta_{1}}=\delta$ instead, where $\delta\neq 0$ is a fixed constant. Here, the conditional Gaussian mean is given by $\delta$ times the right-hand side of~\eqref{eq:cond frac}, whence we obtain from analogous steps 
 \begin{equation}\label{eq:RLcondExpAsymptotics}
     \mathbb{E}[Y_\tau \,|\, Y_{\tau^2}=0, Y_{\tau^3}=\delta]\sim \delta\,\frac{4H(1-2H)}{(2H+1)^2}\,\tau^{1-2H},
 \end{equation}
where $\delta$ may, in principle, also be chosen $\tau$-dependent. Investigating~\eqref{eq:RLcondExpAsymptotics} uncovers similarities
to the increment correlations of the closely related fractional Brownian motion. Conditioned on the increment~$-\delta$ on $[\tau^3, \tau^2]$, \eqref{eq:RLcondExpAsymptotics} essentially describes the asymptotics of the expectation of $Y_{\tau} - Y_{\tau^2}$, since $Y_{\tau^2}=0$. Consequently, for $H\in (0,1/2)$ we can expect an increment with the same sign as~$\delta$, i.e.\ a reversion of the past increment, whereas for $H>1/2$ the trend is expected to continue.

\section{Non-Markov property by reduction via a small-time CLT}\label{subsection: NonMarkovviaCLT}

In this section, we prove the failure of the Markov property for solutions of \eqref{eq:generalSVIE} with general, i.e.\ not necessarily affine, $g(\cdot\hspace{0.05cm};x)$ and Hölder continuous coefficients $b$ and $\sigma$. Our method is based on the reduction to the Gaussian case through a small-time CLT. Motivated by \cite[Theorem 2.3]{FGW24}, let us introduce the following set of assumptions:

\begin{assumption}\label{assumption:coefficientassumptions}
    We suppose $K_b=K_{\sigma}=:K$ and $d=m=1$. The coefficients $b \in C^{\chi_b}(\R)$ and $\sigma \in C^{\chi_{\sigma}}(\R)$ are Hölder continuous with exponents $\chi_{b}$, $\chi_{\sigma}\in (0,1]$. There exist constants $\gamma_* \geq \gamma > 0$ and $C, C^{*}>0$ such that the Volterra kernel $K \in L_{\mathrm{loc}}^2(\R_+)$ satisfies
    \begin{equation}\label{eq:kernelL2order}
        C^{*} \hspace{0.02cm}t^{2\gamma_*}\leq\int_0^t K(s)^2\, \mathrm{d}s \leq C \hspace{0.02cm}t^{2\gamma}, \qquad t \in (0,1].
    \end{equation}
    These constants satisfy the relation
    \begin{equation}\label{eq:varyinginitialCLTexponentcorridor}
        \gamma_* < \min\big\{\gamma + \tfrac{1}{2}, \gamma\hspace{0.05cm}(1+\chi_{\sigma})\big\},
    \end{equation}
    and there exists $\overline{K} \in L_{\mathrm{loc}}^2(\R_+)$ such that for all $0 \leq s \leq t$:
    \[
        \lim_{n \to \infty} \lambda(n)\int_0^{s/n} K( (t-s)/n + r) K(r)\, \mathrm{d}r
        = \int_0^s \overline{K}(t-s + r)\overline{K}(r)\,\mathrm{d}r,
    \]
    where the normalizing sequence $\lambda(n)$ is defined by
    \begin{align}\label{eq: lambda definition}
        \lambda(n) = \left(\int_0^{1/n}K(r)^2\, \mathrm{d}r\right)^{-1},\quad n\in\mathbb{N}.
    \end{align}
\end{assumption}

As Examples \ref{example:limitingkernelmorethanRL} and \ref{example:kernelssatisfying CLT} will show, this assumption is satisfied by a large class of kernels with a similar small-time behavior as the Riemann-Liouville kernel for some parameter $H\in\R_+^*$, e.g.\ gamma and log-modulated fractional kernels. The following is a straightforward extension of the small-time CLT derived in \cite[Theorem 2.3]{FGW24}.

\begin{theorem}\label{theorem: CLT}
    Suppose that Assumption \ref{assumption:coefficientassumptions} holds. For each $x \in \R$, let $g(\cdot\hspace{0.05cm};x): \R_+ \longrightarrow \R$ be a continuous curve with $g(0;x) = x$, such that their collection is locally bounded in $(t,x)$ and satisfies
    \begin{equation}\label{eq:curve_flatness}
        \lim_{n \to \infty} \sup_{|x| \leq R}\sqrt{\lambda(n)}\hspace{0.05cm}\big|g(t/n;x) - x\big| = 0, \qquad \text{for all } t \in \R_+, \ R > 0,
    \end{equation}
    and it holds that
    \begin{align}\label{eq: 9}
        \sup_{t \in (0,1]}\sup_{|x| \leq R}\ t^{-\gamma}|g(t;x) - x| < \infty, \qquad \text{for all }R > 0.
    \end{align}
    Suppose that for each $x \in \R$ there exists a continuous weak solution of
    \begin{equation}\label{eq:SVIEinCLT}
            X_{t} = g(t;x) + \int_0^t K(t-s)\hspace{0.02cm}b(X_s)\, \mathrm{d}s + \int_0^t K(t-s)\hspace{0.02cm}\sigma(X_s)\, \mathrm{d}B_s, \quad t\in\R_+.
    \end{equation} 
    For $(x_n)_{n \geq 1} \subset \R$, let $(\P_{x_n})_{n\ge 1}$ be a family of weak solutions to \eqref{eq:SVIEinCLT} on the canonical path space. If $x_n \longrightarrow x$, then
    \[
        \P_{x_n}\left[ \left(\sqrt{\lambda(n)}\left( X_{t_j/n} - x_n\right)\right)_{j = 1,\dots, N} \in \ \cdot\ \right] \Longrightarrow \P\left[ \big(Y_{t_j}\big)_{j=1,\dots, N} \in \ \cdot\ \right], \qquad n \to \infty,
    \]
    for any $N \geq 1$ and $0 < t_1 < \dots < t_N$, where $Y_t = \sigma(x)\int_0^t \overline{K}(t-s) \,\mathrm{d}B_s$, $t\in\R_+$, is a Gaussian process defined on some filtered space $(\Omega, \mathcal{F}, (\mathcal{F}_t)_{t \in\R_+}, \P)$ such that $B$ is a standard Brownian motion.
\end{theorem}
\begin{proof}
    In view of assumption \eqref{eq:curve_flatness} combined with
     \begin{align*}
        &\ \P_{x_n}\left[ \left(\sqrt{\lambda(n)}\left( X_{t_j/n} - x_n\right)\right)_{j = 1,\dots, N} \in \ \cdot\ \right] 
        \\ &\qquad = \P_{x_n}\left[ \left(\sqrt{\lambda(n)}\left( X_{t_j/n} - g(t_j/n; x_n)\right)\right)_{j = 1,\dots, N} + \left(\sqrt{\lambda(n)}\,\big( g(t_j/n; x_n) - x_n\big)\right)_{j = 1,\dots, N} \in \ \cdot\ \right],
    \end{align*}
    and Slutsky's theorem, it suffices to prove that 
    \[
        \P_{x_n}\left[ \left(\sqrt{\lambda(n)}\left( X_{t_j/n} - g(t_j/n; x_n)\right)\right)_{j = 1,\dots, N} \in \ \cdot\ \right]  \Longrightarrow \P\left[ \big(Y_{t_j}\big)_{j=1,\dots, N} \in \ \cdot\ \right], \qquad n \to \infty.
    \]
    Following along the lines of \cite[Theorem 2.3]{FGW24}, it suffices to establish an analogue of \cite[Proposition~2.2]{FGW24} for \eqref{eq:SVIEinCLT}, i.e.
    \[
        \E_{x}\left[ |Z_{t}|^p \right] \leq \overline{C}_p(x) \left( t^{\frac{p}{2} + p \gamma(1+\chi_b)} + t^{p\gamma(1+\chi_{\sigma})}\right), \qquad p \geq \max\{2/\chi_b, \ 2/\chi_{\sigma}\},
    \]
    where $\overline{C}_p(x)> 0$ is locally bounded in $x$ and $Z$ is given by
    \[
        Z_t = \int_0^t K(t-s)(b(X_s) - b(x))\, \mathrm{d}s + \int_0^t K(t-s)(\sigma(X_s) - \sigma(x))\, \mathrm{d}B_s, \quad t\in\R_+.
    \]
    Such an inequality follows similarly to \cite[Proposition 2.2]{FGW24}, once we have replaced inequality (2.5) therein by 
    \[
        \E_{x}[|X_t - x|^p] \le C_p(x) \hspace{0.02cm}t^{p \gamma}, \qquad t \in [0,1],
    \]    
    where $C_p(x)> 0$ is locally bounded in $x$. The latter follows from $\E_{x}[|X_t - x|^p] \lesssim \E_{x}[|X_t - g(t;x)|^p] + |g(t;x) - x|^p$ combined with assumption \eqref{eq: 9} to bound the second term, while the first term is bounded by \cite[inequality (A.2)]{FGW24} and the local boundedness of $g$ in $(t,x)$. This proves the assertion.
\end{proof}

Let us remark that conditions \eqref{eq:curve_flatness} and \eqref{eq: 9} are satisfied for $g(t;x) = x + t^{\beta}h(x)$ provided that $h$ is locally bounded and $\gamma_* < \beta$. The latter includes $g(t;x) = x$ as a particular case. The following is our main result for this section.

\begin{theorem}\label{theorem:NonMarkovSVIE}
    Suppose that Assumption \ref{assumption:coefficientassumptions} holds. For each $x \in \R$, let $g(\cdot\hspace{0.05cm};x): \R_+ \longrightarrow \R$ be a continuous curve with $g(0;x) = x$, such that their collection is locally bounded in $(t,x)$ and satisfies~\eqref{eq:curve_flatness} and \eqref{eq: 9}. Assume that there exists $x_0 \in \R$ such that $\sigma(x_0) \neq 0$, and that for each $x \in \R$ there exists a continuous weak solution of \eqref{eq:SVIEinCLT} with law $\P_x$ on the canonical path space. If 
    \begin{align}\label{eq: CLT limit}
        Y_t = \int_0^t \overline{K}(t-s) \,\mathrm{d}B_s, \quad t\in\R_+,
    \end{align}
    is not a time-homogeneous Markov process, then also the family of laws $(\P_x)_{x \in \R}$ fails the time-homogeneous Markov property \eqref{eq: Markov timehomogeneous}. 
\end{theorem}
\begin{proof}
    We argue via contradiction. Therefore, let us assume that \eqref{eq:SVIEinCLT} has the time-homogeneous Markov property, i.e.\ $(\P_x)_{x \in \R}$ satisfies~\eqref{eq: Markov timehomogeneous}.  Denote by $p_t(x,A) = \P_x[ X_t \in A ]$, where $t\in\R_+^*$, $x\in\R$ and $A \in\mathcal{B}(\R)$, the associated family of Markov transition probabilities and note that $(p_t)_{t\in\R_+^*}$ defines a family of stochastic transition kernels on $\R\times\mathcal{B}(\R)$, as argued in the introduction.
    
    \textit{Step 1.} Define for every $n\in\mathbb{N}$ a transformation $h_n:\R\longrightarrow\R$ with inverse $h_n^{-1}$ by
    \[
        h_n(y) = \sqrt{\lambda(n)}\hspace{0.02cm}(y - x_0) \ \text{ and } \ h_n^{-1}(x) = x_0 + \lambda(n)^{-1/2}\hspace{0.02cm}x.
    \]
    Since $(\P_x)_{x \in \R}$ fulfills by assumption the Markov property \eqref{eq: Markov timehomogeneous} and $h_n$ is bijective and affine, it follows that also the transformed process $h_n(X_{\cdot/n})$ is under $\mathbb{P}_{x_0}$ a Markov process with respect to its natural filtration for every $n\in\mathbb{N}$, see also \eqref{eq: 2} below. Its transition probabilities are given by 
    \[
        p_t^{(n)}(x,A) := \P_{h_n^{-1}(x)}\big[ h_n(X_{t/n}) \in A\big] = p_{t/n}\big(h_n^{-1}(x), h_n^{-1}(A)\big).
    \] 
    For each $z \in \R$, let us denote by $Y^{\sigma, z}$ the Gaussian process on $(\Omega, \mathcal{F}, (\mathcal{F}_t)_{t \in \R_+}, \P)$ defined by
    \begin{align}\label{eq:RLstartvaluedependence}
        Y_{t}^{\sigma, z}:= z + \sigma(x_{0})\int_{0}^{t}\overline{K}(t-s)\,\mathrm{d}B_s,\quad t\in\R_+.
    \end{align}
    By an application of the small-time CLT provided in Theorem \ref{theorem: CLT}, we obtain
    \begin{align}\label{eq: convergence fdd x0}
        \P_{x_0}\big[ h_n\big(X_{\cdot/n}\big) \in \ \cdot \ \big] \Longrightarrow \P \big[Y^{\sigma, 0}\in \ \cdot \ \big]
    \end{align}

    for their finite-dimensional distributions. Set $q_t(x,\mathrm{d}y) = \P[ Y_t^{\sigma, x} \in \mathrm{d}y]$. L\'evy's continuity theorem implies that these kernels are weakly continuous. Moreover, it follows from the Gaussianity in \eqref{eq:RLstartvaluedependence} that $q_t(x,\mathrm{d}y) = \delta_x \ast q_t(0, \mathrm{d}y)$ and hence $x \longmapsto \int_{\R} f(y) \,q_t(x, \mathrm{d}y) = \int_{\R}f(x+y)\, q_t(0, \mathrm{d}y)$ is bounded and Lipschitz continuous whenever $f$ is bounded and Lipschitz continuous. In the subsequent steps, we show that $(q_t)_{t\in\R_+^*}$ forms a family of Markov transition kernels on $\R\times\mathcal{B}(\R)$, and that the finite-dimensional distributions of $Y^{\sigma, 0}$ are necessarily given via the Chapman-Kolmogorov equations determined by the kernels $(q_t)_{t\in\R_+^*}$. In particular, since the finite-dimensional distributions uniquely determine the continuous process $Y^{\sigma, 0}$, it necessarily has to be a Markov process. However, due to $\sigma(x_0)\neq 0$, this contradicts the assumed non-Markovianity of $Y=\overline{K}\ast \mathrm{d}B$, which will eventually prove our assertion.

    \textit{Step 2.} Let us first show that for all $t, M > 0$ and each bounded and Lipschitz continuous function $f$ it holds
    \begin{align}\label{eq: N1 convergence}
        \lim_{n \to \infty}\sup_{|x| \leq M}\left| \int_{\R}f(y)\, p_t^{(n)}\big( x ,\mathrm{d}y\big) - \int_{\R}f(y)\, q_t(x,\mathrm{d}y)\right| = 0.
    \end{align}
    Let $F_{f,n}(x) = \int_{\R}f(y)\, p_t^{(n)}\big( x ,\mathrm{d}y\big) - \int_{\R}f(y)\, q_t(x,\mathrm{d}y)$. Since $|F_{f,n}| \leq 2 \|f\|_{\infty}$, we find for each $n \geq 1$ some $|z_n| \leq M$ such that
    \[
      \big|F_{f,n}(z_n)\big| > \sup_{|x| \leq M}\big|F_{f,n}(x)\big| - \frac{1}{n}.
    \]
    Therefore, \eqref{eq: N1 convergence} is implied by
    \begin{equation}\label{eq: N1 convergence 2}
        \lim_{n\to\infty} \big|F_{f,n}(z_{n})\big|=0.
    \end{equation}
    Since $\big(|F_{f,n}(z_{n})|\big)_{n\in\mathbb{N}}$ is a bounded sequence, it is sufficient for the latter to prove that every convergent subsequence with indices $(n_{k})_{k\in\mathbb{N}}$ has limit $0$. Without loss of generality, we may assume that also $(z_{n_{k}})_{k\in\mathbb{N}}$ converges to some $z \in [-M,M]$, because, due to the compactness of $[-M,M]$, one could otherwise pass to a convergent subsequence $\big(z_{n_{k_{l}}}\big)_{l\in\mathbb{N}}$, whose image converges to the same limit as that of the original subsequence. For this choice, we can estimate
    \begin{equation*}
    \begin{aligned}
        \lim_{k\to\infty} \big|F_{f,n_{k}}(z_{n_{k}})\big|\le &\lim_{k\to\infty} \left| \int_{\R}f(y)\, p_t^{(n_{k})}\big( z_{n_{k}},\mathrm{d}y\big) -\int_{\R}f(y)\, q_t(z,\mathrm{d}y)\right| \\
        &+ \lim_{k\to\infty} \left|\int_{\R}f(y)\, q_t(z,\mathrm{d}y) - \int_{\R}f(y)\, q_t(z_{n_{k}},\mathrm{d}y)\right|.
    \end{aligned}
    \end{equation*}
    Observe that the second summand is zero due to the weak continuity of the kernel~$q_{t}$ as shown at the end of Step 1. For the first part, note that $h_{n_k}(X_{t/n_k}) = z_{n_k} + \sqrt{\lambda(n_k)}\hspace{0.02cm}\big(X_{t/n_k} - h^{-1}_{n_k}(z_{n_k})\big)$. Hence, we may write for a bounded, globally Lipschitz continuous function $f$: 
    \begin{align*}
        \int_{\R}f(y)\, p_t^{(n_{k})}\big( z_{n_{k}},\mathrm{d}y\big)
        &= \E_{h_{n_k}^{-1}(z_{n_k})}\big[ f(h_{n_k}(X_{t/n_k})) \big]
        \\ &= \E_{h_{n_k}^{-1}(z_{n_k})}\left[ f\Big( z + \sqrt{\lambda(n_k)}\hspace{0.02cm}\big(X_{t/n_k} - h_{n_k}^{-1}(z_{n_k})\big)\Big) \right] 
        \\ &\qquad + \E_{h_{n_k}^{-1}(z_{n_k})}\left[ f\Big( z_{n_k} + \sqrt{\lambda(n_k)}\hspace{0.02cm}(X_{t/n_k} - h_{n_k}^{-1}(z_{n_k})\big)\Big) \right] 
        \\ &\qquad - \E_{h_{n_k}^{-1}(z_{n_k})}\left[ f\Big( z + \sqrt{\lambda(n_k)}\hspace{0.02cm}\big(X_{t/n_k} - h_{n_k}^{-1}(z_{n_k})\big)\Big) \right]. 
    \end{align*}
    For the first term, let us note that $h_{n_k}^{-1}(z_{n_k})\longrightarrow x_0$ since $(z_{n_k})_{k\in\mathbb{N}}$ is a bounded sequence and $\lambda(n_k)^{-1/2} \longrightarrow 0$. Thus, an application of the small-time CLT given in Theorem \ref{theorem: CLT} yields
    \[
        \lim_{k \to \infty} \E_{h_{n_k}^{-1}(z_{n_k})}\left[ f\Big( z + \sqrt{\lambda(n_k)}\hspace{0.02cm}\big(X_{t/n_k} - h_{n_k}^{-1}(z_{n_k})\big)\Big) \right] = \E\big[f\big(z + Y_t^{\sigma,0}\big)\big]
        = \int_{\R}f(y)\, q_t(z,\mathrm{d}y).
    \]
     The remaining terms converge to zero since $f$ is Lipschitz continuous and $z_{n_k} \longrightarrow z$. Thus we have shown that $\lim_{k\to\infty} \big|F_{f,n_{k}}(z_{n_{k}})\big| = 0$, which proves \eqref{eq: N1 convergence 2} and hence also \eqref{eq: N1 convergence}. 
    
    \textit{Step 3.} Let $N \geq 1$, $f_1,\dots, f_N: \R \longrightarrow \R$ be bounded and Lipschitz, and $0=t_0 < t_1 < \dots < t_N$. Then, by the convergence of the finite-dimensional distributions shown in \eqref{eq: convergence fdd x0}, we obtain
    \begin{equation}\label{eq:NonMarkovCLTfindimdistributions}
        \E_{x_0}\big[ f_1\big( h_n(X_{t_1/n})\big) \cdots f_N\big( h_n(X_{t_N/n})\big) \big] \longrightarrow \E\big[ f_1\big(Y^{\sigma, 0}_{t_1}\big) \cdots f_N\big(Y^{\sigma, 0}_{t_N}\big)\big],
    \end{equation}
    i.e.\ the finite-dimensional distributions of $Y^{\sigma, 0}$ are determined as the limit of the left-hand side. Let us now show that we also have 
    \begin{align}\label{eq: 1cltMarkov}
        I_{n}:=\E_{x_0}\left[ f_1\big( h_n(X_{t_1/n})\big)\cdots f_N\big( h_n( X_{t_N/n})\big)\right] \longrightarrow I,
    \end{align}
    where the limit $I$ is given by
    \[
        I = \int_{\R}\cdots \int_{\R} f_1(x_1)\cdots f_N(x_N) \,q_{t_{N} - t_{N-1}}(x_{N-1}, \mathrm{d}x_N)\cdots q_{t_1-t_0}(0, \mathrm{d}x_1).
    \] 
    For $N = 1$, this assertion follows from \eqref{eq: convergence fdd x0}. Therefore, let us suppose that $N \geq 2$, and define for $j \in\{2, \dots, N\}$ the measures
    \begin{align*}
        P_j^{(n)}(\mathrm{d}x_{j}, \dots, \mathrm{d}x_2, x_1) 
        &= p^{(n)}_{t_{j} - t_{j-1}}\big( x_{j-1}, \mathrm{d}x_{j}\big)\cdots p_{t_2 - t_1}^{(n)}\big( x_1, \mathrm{d}x_2\big).
    \end{align*}
    An application of \cite[Proposition III.1.4]{revuzyor} for the, with respect to its natural filtration, time-homogeneous Markov process $h_n(X_{\cdot/n})$ under $\mathbb{P}_{x_0}$ whose transition kernels are given by $p_{t}^{(n)}$, $t\in\R_+^*$, yields for $I_n$:
    \begin{align}\label{eq: 2}
        &\E_{x_0}\big[ f_1\big( h_n(X_{t_1/n})\big)\cdots f_N\big( h_n(X_{t_N/n})\big)\big] 
        \\ &\quad = \int_{\R}\cdots \int_{\R} f_1(x_1)\cdots f_N(x_N)\, P_N^{(n)}(\mathrm{d}x_N, \dots, \mathrm{d}x_2, x_1) \,p^{(n)}_{t_1}(0, \mathrm{d}x_1). \notag
    \end{align} 
    For arbitrary $N \geq 2$, we can derive the estimate
    \[
        \left| \E_{x_0}\big[ f_1\big( h_n(X_{t_1/n})\big)\cdots f_N\big( h_n(X_{t_N/n})\big)\big]  - I\right| \leq \sum_{j=1}^N J_j^{(n)}
    \]
    by inductively replacing the transition kernel in each dimension step-by-step and finally applying the triangle inequality, where the summands $J_j^{(n)}$ are given by
    \begin{align*}
        J_j^{(n)} &= \bigg|\int_{\R^{j-1}}\int_{\R} Q_j(x_j)  \left(p^{(n)}_{t_j - t_{j-1}}\big( x_{j-1}, \mathrm{d}x_j\big) - q_{t_j - t_{j-1}}(x_{j-1},\mathrm{d}x_j)\right)
        \\ &\hskip22mm \cdot \bigg( \prod_{k=1}^{j-1}f_{k}(x_k)\bigg) P_{j-1}^{(n)}(\mathrm{d}x_{j-1}, \dots \mathrm{d}x_2, x_1) \hspace{0.03cm}p^{(n)}_{t_1}(0, \mathrm{d}x_1) \bigg|
    \end{align*}
    when $j \in\{ 2, \dots, N\}$, and for $j = 1$ we define
    \[
        J_1^{(n)} = \left| \int_{\R} Q_1(x_1)\left( p_{t_1}^{(n)}(0, \mathrm{d}x_1) - q_{t_1}(0, \mathrm{d}x_1)\right) \right|.
    \]   
    Here, the functions $Q_j$ are defined by $Q_N=f_N$ and for $j\in\{1, \dots, N-1\}$ by
    \begin{align*}
        Q_{j}(x_j) &= \int_{\R^{N-j}} \bigg(\prod_{k=j}^N f_k(x_k) \bigg)\hspace{0.02cm} q_{t_{N} - t_{N-1}}(x_{N-1}, \mathrm{d}x_N)\cdots q_{t_{j+1} - t_j}(x_j, \mathrm{d}x_{j+1}).
    \end{align*}
    Since by Step 1, integrals with respect to $q_t(x,\cdot)$ are as functions in $x$ stable for bounded Lipschitz functions and every $f_{k}, k\in\{1,\dots,N\}$, is bounded and Lipschitz continuous, one can inductively show that also $Q_j$ is bounded and Lipschitz continuous. Moreover, for each fixed $j \in\{ 2, \dots, N\}$ we can choose for any $\varepsilon > 0$ a large enough constant $M > 0$ such that $p^{(n)}_{t_{j-1}}(0, [-M,M]^c)<\varepsilon$ holds uniformly in $n\ge 1$. This is indeed possible since $p_t^{(n)}( x, \cdot) \Longrightarrow q_t(x, \cdot)$ holds for all $x\in\R$ and $t>0$ by \eqref{eq: N1 convergence}, whence this sequence of measures is tight. Therefore, we have   
    \begin{align}\label{eq: N1 tightness}
        \lim_{M \to \infty}\sup_{n \geq 1}\ p^{(n)}_t\big( x, [-M,M]^c\big) = 0,
    \end{align}
    in particular for the choice $x=0$ and $t=t_{j-1}$, which implies that we can select $M>0$ as claimed before. For this $M$ we can choose $n\ge 1$ large enough such that
    \begin{displaymath}
        \sup_{|x_{j-1}| \leq M}\left| \int_{\R}Q_{j}(x_{j})\left(p_{t_{j}-t_{j-1}}^{(n)}\big( x_{j-1},\mathrm{d}x_{j}\big) - 
        q_{t_{j}-t_{j-1}}(x_{j-1},\mathrm{d}x_{j})\right)\right|<\varepsilon
    \end{displaymath}
    holds by \eqref{eq: N1 convergence}. Hence, by combining this with the boundedness of $Q_{j}$ and $f_{k}, k\in\{1,\dots,N\},$ as well as the Chapman-Kolmogorov equations for the transition kernels $\big(p^{(n)}_t\big)_{t\in\R_+^*}$ associated with the time-homogeneous Markov process $h_n( X_{\cdot/n})$, we can estimate
    \begin{equation*}
    \begin{aligned}
        J_j^{(n)} &\leq C_{1}\int_{\R^{j-1}}\sup_{x_{j-1}\in\mathbb{R}}\bigg|\int_{\R} Q_j(x_j)  \left(p^{(n)}_{t_j - t_{j-1}}\big( x_{j-1}, \mathrm{d}x_j\big) - q_{t_j - t_{j-1}}(x_{j-1},\mathrm{d}x_j)\right)\\
        &\hskip40mm \cdot\1_{[-M,M]}(x_{j-1})\bigg|\, P_{j-1}^{(n)}(\mathrm{d}x_{j-1}, \dots \mathrm{d}x_2, x_1)\, p^{(n)}_{t_1}(0, \mathrm{d}x_1) \\
        &\hskip4.5mm+ C_{2}\int_{\R^{j-1}} \1_{[-M,M]^c}(x_{j-1}) \,P_{j-1}^{(n)}(\mathrm{d}x_{j-1}, \dots \mathrm{d}x_2, x_1) \,p^{(n)}_{t_1}(0, \mathrm{d}x_1)\\
        &\leq C_{3}\Big(\varepsilon\, p^{(n)}_{t_{j-1}}(0, \mathbb{R})+\sup_{n\ge 1}p^{(n)}_{t_{j-1}}\big(0, [-M,M]^c\big)\Big)\\
        &= C_{4}\,\varepsilon,
    \end{aligned}
    \end{equation*}
     where $C_i$, $i\in\{1,\dots,4\}$, denote suitable positive constants independent of $\varepsilon$ and $n$. Therefore, we can conclude $J_j^{(n)} \longrightarrow 0$ for every $j \in\{ 2, \dots, N\}$. Finally, when $j=1$, the desired convergence $J_1^{(n)} \longrightarrow 0$ follows directly from the weak convergence observed in \eqref{eq: N1 convergence}. This proves the conjectured convergence $I_n \longrightarrow I$ as $n\to\infty$.
     
     \textit{Step 4.} Let us now show that $(q_t)_{t\in\R_+^*}$ corresponds to a consistent family of Markov transition kernels, i.e.\ it fulfills the Chapman-Kolmogorov equations. Firstly, consider a bounded Lipschitz function $f: \R \longrightarrow \R$, and $0 < s < t$. Then we obtain from \eqref{eq: 1cltMarkov},~\eqref{eq: 2}, and $\big(p_t^{(n)}\big)_{t\in\R_+^*}$ being Markov transition kernels:
     \begin{align}\label{eq: 7}
         \int_{\R} \left(\int_{\R} f(x_2) \,q_{t-s}(x_1,\mathrm{d}x_2) \right)q_{s}(0, \mathrm{d}x_1) \notag
         &= \lim_{n \to \infty} \int_{\R} \int_{\R} \1_{\R}(x_1)\hspace{0.02cm}f(x_2) \hspace{0.02cm}p^{(n)}_{t-s}(x_1,\mathrm{d}x_2) \hspace{0.02cm}p_{s}^{(n)}(0, \mathrm{d}x_1) \notag
         \\ &= \lim_{n \to \infty} \int_{\R} f(x_2) \, p_{t}^{(n)}(0, \mathrm{d}x_2)
         \\ \notag &= \int_{\R} f(x_2) \,q_{t}(0, \mathrm{d}x_2).
     \end{align}
     As we have so far only considered the initial law $\delta_{0}$, let us extend this observation to general starting points $z \in \R$. First, note that by \eqref{eq:RLstartvaluedependence}, we have $q_t(z, \cdot) = q_t(0,\cdot) \ast \delta_z$, and hence also $q_t(x + z,\cdot) = q_t(x,\cdot) \ast \delta_z$. Using these relations, we obtain
     \begin{align*}
         \int_{\R} \left(\int_{\R} f(x_2) \,q_{t-s}(x_1,\mathrm{d}x_2) \right)\,q_{s}(z, \mathrm{d}x_1)
         &= \int_{\R} \int_{\R} \int_{\R} f(x_2) \,q_{t-s}(x_1 + y,\mathrm{d}x_2) \,q_{s}(0, \mathrm{d}x_1)\,\delta_z(\mathrm{d}y)
         \\ &= \int_{\R} \int_{\R} f(x_2) \,q_{t-s}(x_1 + z,\mathrm{d}x_2) \,q_{s}(0, \mathrm{d}x_1)
         \displaybreak[1]\\ &= \int_{\R} \int_{\R} f(x_2+z) \,q_{t-s}(x_1,\mathrm{d}x_2) \,q_{s}(0, \mathrm{d}x_1)
         \\ &= \int_{\R} f(x_2+z) \,q_{t}(0,\mathrm{d}x_2)
         \\ &= \int_{\R} f(x_2) \,q_{t}(z,\mathrm{d}x_2),
     \end{align*}
     where the second-to-last step follows from the previous computation \eqref{eq: 7}. Since $f$ is an arbitrary bounded and Lipschitz continuous function, one may show by a standard approximation procedure of indicator functions by bounded Lipschitz functions that we arrive at the relation $\int_{\R}q_{t-s}(x_1,\cdot)\, q_s(z, \mathrm{d}x_1) = q_{t}(z,\cdot)$, which shows that $(q_t)_{t\in\R_+^*}$ is a family of consistent Markov transition kernels by \cite[Definition III.1.2]{revuzyor}. Combining this with $Y^{\sigma, 0}_0=0$ and 
     \begin{displaymath}
         \E\left[ f_1\big(Y^{\sigma, 0}_{t_1}\big) \cdots f_N\big(Y^{\sigma, 0}_{t_N}\big)\right]=\int_{\R}\cdots \int_{\R} f_1(x_1)\cdots f_N(x_N) \,q_{t_{N} - t_{N-1}}(x_{N-1}, \mathrm{d}x_N)\cdots q_{t_1}(0, \mathrm{d}x_1),
     \end{displaymath}
     according to Step 3, which generalizes again by classical approximation arguments to nonnegative Borel measurable functions $f_1,\dots,f_N$, shows via \cite[Proposition III.1.4]{revuzyor} that $Y^{\sigma,0}=\sigma(x_0)\big(\overline{K}\ast\mathrm{d}B\big)$ is a Markov process with its time-homogeneous transition family given by $(q_t)_{t\in\R_+^*}$. By $\sigma(x_0)\neq 0$, this contradicts the assumed non-Markovianity of $Y = \overline{K}\ast\mathrm{d}B$. Hence, the proof is completed.
\end{proof}

We proceed by providing several examples of kernel families which are covered by the above framework. In both cases, let us suppose that we are given an admissible $g$, $b \in C^{\chi_b}(\R)$ and $\sigma \in C^{\chi_{\sigma}}(\R)$ with $\chi_b, \chi_{\sigma} \in (0,1]$, so that we may focus entirely on the Volterra kernel. Our first example covers kernels that exhibit an equivalent small-time behavior as classical Riemann-Liouville kernels, see also \cite[Example 2.9]{FGW24} for additional details. 

\begin{example}\label{example:limitingkernelmorethanRL}
    Suppose that for $K\in L_{\mathrm{loc}}^2(\R_+)$ there exist constants $H, C(H) >0$ such that
    \[
       K(t) \sim C(H)\hspace{0.03cm}t^{H-1/2}, \quad \  \mbox{as} \ t \searrow 0.
    \]
    Then we obtain $\gamma = \gamma_* = H$, $\lambda(n) \sim 2H\hspace{0.03cm}C(H)^{-2}\hspace{0.03cm}n^{2H}$, and $\overline{K}(t) = \sqrt{2H}\hspace{0.03cm}t^{H - 1/2}$ follows from the dominated convergence theorem. Therefore, Assumption \ref{assumption:coefficientassumptions} is satisfied. It follows from Lemma~\ref{lemma:RLnonMarkov} that $Y = \overline{K} \ast \mathrm{d}B$ does not satisfy the Markov property for $H \neq 1/2$, whence Theorem \ref{theorem:NonMarkovSVIE} is applicable.
    
    However, note that by taking $H = 1/2$ this class also covers regular kernels $K \in C^1(\R_+)$ with $K(0)\in~\hspace{-0.08cm}\R_+^*$. For all such regular kernels, we obtain $\overline{K}\equiv 1$. However, it is clear that $Y = \overline{K} \ast \mathrm{d}B = B$ is a time-homogeneous Markov process, whence Theorem \ref{theorem:NonMarkovSVIE} is not applicable for $H = 1/2$. 
\end{example}

Our second example covers kernels that are not asymptotically equivalent to the Riemann-Liouville kernel in the small-time regime.

\begin{example}\label{example:kernelssatisfying CLT}
    The $\log$-modulated fractional kernel with parameter $H>0$ defined by
    \[
            K(t) = \frac{t^{H - 1/2}}{\Gamma(H+1/2)}\log(1 + 1/t), \quad t\in\R_+^*,
    \]
    satisfies \eqref{eq:kernelL2order} with $\gamma_*=H$ and arbitrary $\gamma \in (0,H)$. Moreover, a short computation and an application of Karamata's theorem (see \cite[Proposition~1.5.10]{BiGoTe87}) shows that $\overline{K}(t)=\sqrt{2H}\hspace{0.02cm}t^{H-1/2}$, see \cite[Example 2.10]{FGW24} for details. Hence, Assumption \ref{assumption:coefficientassumptions} is satisfied, and Theorem \ref{theorem:NonMarkovSVIE} is applicable whenever $H\neq 1/2$. 
\end{example}

 Although the framework of this section permits considerable generality in the coefficients, in particular the drift, the lower bound in \eqref{eq:kernelL2order} should be regarded as an additional condition beyond the standard assumptions for weak existence. While this bound is satisfied by the important kernels considered above, it fails for certain more exotic kernels exhibiting exponential decay in the small-time regime. For example, $K(t)=\frac{1}{t}\hspace{0.02cm}\exp(-\frac{1}{2t})$ cannot be bounded from below in $L^2$ by a power function since $\int_0^t K(s)^2 \,\mathrm{d}s= \exp(-\frac{1}{t})$. Still, as the kernel is locally Lipschitz on $\R_+$, weak existence to \eqref{eq:generalSVIE} with $K_b=K_{\sigma}=K$ can be established by \cite[Theorem 1.2 and Theorem 6.1]{weak_solution}, and thus the failure of the Markov property~\eqref{eq: Markov timehomogeneous} can be shown via the results from Section \ref{subsection:newmomentformula} for affine $b$ and $g$ of the form \eqref{eq: g affine}.
 
Finally, note that since our proof of Theorem \ref{theorem:NonMarkovSVIE} relies on the small-time CLT derived in \cite{FGW24} for the initial time $t = 0$, we are necessarily restricted to the time-homogeneous case. An extension towards proving also the failure of the time-inhomogeneous Markov property would require an analogue of \cite[Theorem 2.3]{FGW24} for a (conditional) small-time CLT at arbitrary time $t\in\R_+$.

Beyond this restriction to the time-homogeneous Markov property, the small-time CLT approach employed here provides only a \emph{sufficient} condition for non-Markovianity. A broader class of sufficient conditions ensuring the failure of even the time-inhomogeneous Markov property, covering multidimensional kernels $K_b$ and $K_{\sigma}$ as well as $\mathcal{F}_0$-measurable inputs $g$, is established in~\cite{FGW25} using perturbation methods through Hilbert space-valued Markovian lifts. The focus of the present paper is instead on \emph{characterizing} non-Markovianity within affine frameworks and clarifying its connection with the non-Markovian behavior of Gaussian Volterra processes in the small-time regime.

\vspace{0.3cm}
\noindent{\bf Acknowledgement.} We thank Mathias Beiglb\"ock and Paul Kr\"uhner for helpful comments.

\begin{footnotesize}
\bibliographystyle{siam}
\bibliography{literature}
\end{footnotesize}

\end{document}